\documentclass[11pt,a4]{article}
\usepackage{amsmath,amssymb,amsfonts,amsthm}
\usepackage{enumerate}
\usepackage{tasks}
\usepackage{threeparttable}
\usepackage{float}
\usepackage{longtable}
\usepackage{bussproofs}
\usepackage{hyperref}
\usepackage{multirow}
\usepackage{cleveref}
\usepackage{xcolor}
\hypersetup{
colorlinks=true,
linkcolor=red, 
citecolor={blue},     
urlcolor=black,
pagebackref=true,
}
\makeatletter
\newcommand*{\rom}[1]{\expandafter\@slowromancap\romannumeral #1@}
\makeatother
\usepackage[left=2.3cm, right=2.3cm, top=2.5cm, bottom=2.5cm]{geometry}
\newcommand{\D}{\mathcal{D}}

\newcommand{\AXC}{\AxiomC}
\newcommand{\UIC}{\UnaryInfC}
\newcommand{\BIC}{\BinaryInfC}

\newcommand{\DP}{\DisplayProof}
\newcommand{\RL}{\RightLabel}
\newcommand{\LL}{\LeftLabel}
\newtheorem{The}{Theorem}[section]

\newtheorem{Lem}[The]{Lemma}
\newtheorem{Cor}[The]{Corollary}
\newtheorem{Def}[The]{Definition}

\newtheorem{Examp}[The]{Example}

\let\oldproofname=\proofname
\renewcommand{\proofname}{\textit{\rm\bf\oldproofname}}
\title{\bf\Large Rooted Hypersequent Calculus for  Modal  Logic S5
\thanks {2010{ \it Mathematics Subject Classification}.  Primary, 03F05, 03B45.}}
\author{{\bf Mojtaba Aghaei$^{{\rm a}}$   and~ {\bf Hamzeh Mohammadi}$^{{\rm a}}$}\thanks{Corresponding author.} \\
{\small{ $^{{\rm a}}$Department of Mathematical Sciences,  Isfahan University of Technology}}\vspace{-1mm}\\
{\small{     Isfahan,   Iran}}\\
{\small{aghaei@cc.iut.ac.ir}}\vspace{-1mm}\\
{\small{hamzeh.mohammadi@math.iut.ac.ir}}\vspace{-1mm}
}
\date\today

\begin{document}
  \maketitle

\begin{abstract}
We  present  a rooted hypersequent  calculus   for  modal propositional logic S5. We show that all rules of this calculus are invertible and that the  rules of weakening, contraction, and  cut  are admissible. Soundness and completeness are established as well.
\end{abstract}
\section{\bf Introduction}
The propositional modal logic {\sf S5}   is one of the peculiar  modal logics in several respects. Most notably from the proof-theoretical point of view, S5 has so far resisted all efforts to provide it with a acceptable cut-free sequent calculus. Whereas the framework of sequent calculi has proven quite successful in providing analytic calculi for a
number of normal modal logics such as {\sf K}, {\sf KT} or {\sf S4} \cite{wansing1994sequent}. For some formats of rules it can even be shown that no such calculus can exist \cite{Lellmann-Pattinson}. Perhaps, the easiest way of demonstrating this resistance is Euclideanness axiom: $ {\sf (5)} \,\, \Diamond A\rightarrow \Box\Diamond A $. Sequent calculus systems for S5   have been widely studied for a long time. Several  authors have introduced
many sequent calculi   for S5,  including Ohnishi and Matsumoto \cite{ohnishi1959gentzen}, Mints \cite{mints1997indexed}, Sato \cite{sato1980cut}, Fitting \cite{Fitting}, Wansing \cite{wansing1994sequent} and Bra\"uner \cite{brauner2000cut}. The efforts to develop   sequent calculus to accommodate cut-free systems
for  S5   leading to introduce a variety of new sequent framework. Notably, labelled sequent calculus (see e.g. \cite{negri2005proof}), double sequent calculus (see e.g. \cite{indrzejczak1998cut}), display calculus
(see e.g. \cite{ belnap1982display, wansing1999predicate}),  deep inference system (see e.g.  \cite{stouppa2007deep}), nested sequent 
(see \cite{brunnler2009deep, poggiolesi-Gentzen Calculi}), 
hypersequent calculus,  which was introduced independently in \cite{Avron hyper, mints1974, Pottinger}  and finally, grafted hypersequents (\cite{Kuznets Lellmann}), which  combines the formalism of nested sequents with that of hypersequents. Hypersequent calculus provided numerous cut-free formulations for the logic S5, including
Pottinger \cite{Pottinger}, Avron \cite{Avron hyper}, Restall \cite{restall2005proofnets}, Poggiolesi \cite{poggiolesi2008cut}, Lahav \cite{Lahav}, Kurokawa \cite{kurokawa2013hypersequent},  Bednarska et al \cite{Bednarska Indrzejczak}, and Lellmann  \cite{ Lellmann}. 

The aim of this paper is to introduce a new  sequent-style  calculus  for S5  by suggesting a
framework of rooted hypersequents. A rooted hypersequent is of the form $ \Gamma\Rightarrow\Delta\,||\, P_1\Rightarrow Q_1\,|\cdots|\, P_n\Rightarrow Q_n $, where $ \Gamma $ and $ \Delta $ are multisets of arbitrary formulas and  $ P_i $ and $ Q_i $ are  multisets of atomic formulas. Precisely,  a rooted hypersequent
is given by a sequent $ \Gamma\Rightarrow\Delta $, called its root, together with a hypersequent $ \cal H $, called its crown, where all formulas in the components of crown are atomic formulas. The sequents in the crown  work as storage for  atomic formulas that    they might be  used  to get axioms. This sequent  is inspired by the  grafted hypersequent in \cite{Kuznets Lellmann}. A difference is that in the grafted hypersequent, all formulas in the crown can be arbitrary formula. Thus, the notion of our calculus is very close to the notion of grafted hypersequent.
The main idea for constructing rooted hypersequent is to take an ordinary sequent $ \Gamma\Rightarrow\Delta $ as a
root  and add  sequences of multisets of atomic formulas to it.

 Our calculus   has the subformula property,   and we show that all rules of this calculus are invertible and that the  rules of weakening, contraction, and  cut are admissible. It is worth pointing out that in order to prove admissibility of cut rule, we make use of a normal form, called Quasi Normal Form, which   is based on using modal  and negation of modal  formulas as literals. Soundness and completeness are established as well.

We proceed as follows. In the next section  we recall the modal logic {\sf S5}. In Section \ref{G3S5}, we present rooted hypersequent  calculus  $ {\cal R}_{\sf S5}$. In Section \ref{section Soundness}, we prove  soundness of the system with respect to  Kripke models. In  Section \ref{Structural properties}, we prove the admissibility of  weakening and contraction rules,   and some other properties of $ {\cal R}_{\sf S5}$. In Section \ref{sec cut}, we prove admissibility of cut rule, and completeness of the system. Finally, we conclude the paper in Section \ref{conclution}.
\section{\bf  Modal logic S5} \label{sec 2}
In this section, we recall  the  axiomatic formulation of the   modal logic  S5.

The language of  modal logic S5 is obtained by adding to the language of
propositional logic the two modal operators $ \Box $ and $ \Diamond $. Atomic formulas are denoted by 
$ p,q,r, $ and so on.
 Formulas, denoted by 
$ A,B,C,\ldots $,
 are defined by the following grammar:
\[ A:=\bot\,|\top\,|p\,|\neg A|\,A\wedge A|\,A\vee A|\,A\rightarrow A|\,\Diamond A|\,\Box A,\]
where $ \bot $ is a  constant for falsity, and $ \top  $ is a  constant for truth.

Modal logic S5 has the following axiom schemes:
\begin{align*}
& \text{All propositional tautologies,}\\
& \text{(Dual)} \quad \Box A\leftrightarrow \neg \Diamond \neg A,\\
&\text{(K)}\quad\Box(A\rightarrow B)\rightarrow(\Box A\rightarrow \Box B),\\
&\text{(T)}\quad  \Box A\rightarrow A,\\
&\text{(5)} \quad \Diamond A\rightarrow \Box\Diamond A.
\end{align*}
Equivalently, instead of (5) we can use:
\begin{align*}
& \text{(4)}\quad \Box A\rightarrow \Box\Box A,\\
& \text{(B)} \quad  A\rightarrow \Box\Diamond A.
\end{align*}
The proof rules are Modus Ponens and  Necessitation:
\begin{center}
	\AxiomC{$ A $}
	\AxiomC{$ A\rightarrow B $}
	\RightLabel{ MP,}
	\BinaryInfC{$ B $}
	\DisplayProof
	$\quad$
	\AxiomC{$A$}
	\RightLabel{N.}
	\UnaryInfC{$\Box A$}
	\DisplayProof
\end{center}
Rule Necessitation  can be applied only to theorems (i.e. to formulas derivable from no premise),  for a detailed exposition see \cite{Chellas, Blackburn}.
 If 
$ A $ is derivable in S5 from   assumption
$ \Gamma $, we   write $ \Gamma\vdash_{\text{S5}} A $. 
\section{ Rooted Hypersequent $ \cal{R}_{\sf S5} $}\label{G3S5}
Our calculus is based on finite multisets, i.e. on sets counting multiplicities of elements. We use certain categories of letters, possibly with subscripts or
primed, as metavariables for certain syntactical categories (locally different
conventions may be introduced):
\begin{tasks}(2)
		\task[$ \bullet $]
	$ p$ and $q$ for atomic formulas, 
		\task[$ \bullet $]
	$P$ and $Q$ for multisets of atomic formulas,
	 \task[$ \bullet $]
	 	$ M$ and $N$ for multisets of modal formulas,
	\task[$ \bullet $]
	$ \Gamma$ and $\Delta$ for multisets of arbitrary formulas.
\end{tasks}
In addition, we  use the following notations.
\begin{itemize}
\item 
The union of multisets  $ \Gamma $ and $ \Delta $  is
indicated simply by $ \Gamma,\Delta $. The union of a multiset $ \Gamma $
with a singleton multiset $ \{A\} $ is written  $ \Gamma, A $.
\item
 We use $\neg\Gamma$ for multiset of formulas $\neg A$ such that $A\in \Gamma$.
\end{itemize}
\begin{Def}
	A sequent is a pair of multisets $ \Gamma $ and $ \Delta $, written as $ \Gamma\Rightarrow\Delta $.  A hypersequent is a multiset of sequents,
	written $ \Gamma_1\Rightarrow\Delta_1\,|\,\cdots,| \Gamma_n\Rightarrow\Delta_n $, where each $ \Gamma_i\Rightarrow\Delta_i $ is called a component. A rooted hypersequent
	is given by a sequent $ \Gamma\Rightarrow\Delta $, called its root, together with a hypersequent $ \cal H $, called its crown, where all formulas in the components of crown are atomic formulas,
	and is written as $ \Gamma\Rightarrow\Delta\,||\, {\cal H} $. If the crown is the empty hypersequent, the double-line separator can be omitted: a rooted hypersequent $ \Gamma\Rightarrow\Delta $ is understood as $ \Gamma\Rightarrow\Delta\,||\, \emptyset $. Formulas occurring on the
	left-hand side of the sequent arrow in the root or a component of the crown are called antecedent formulas; those occurring on the right-hand side succedent formulas.
\end{Def}
Therefore, the notion of a rooted hypersequent 
\[\Gamma\Rightarrow\Delta\, ||\, P_1 \Rightarrow Q_1\,|\, P_2\Rightarrow Q_2\,|\, \ldots \,|\,P_n\Rightarrow Q_n,\]
 can be seen as a restriction of the notion of grafted hypersequent as in \cite{Kuznets Lellmann}.
\begin{Def}
Let $\Gamma\Rightarrow\Delta\, ||\, P_1 \Rightarrow Q_1\,|\, P_2\Rightarrow Q_2\,|\, \ldots \,|\,P_n\Rightarrow Q_n $ be a rooted hypersequent. Its formula interpretation is the
formula 
$\bigwedge\Gamma\rightarrow \bigvee\Delta\vee\bigvee\limits_{i=1}^n\Box(\bigwedge P_i\rightarrow \bigvee Q_i)$.
\end{Def}

 The axioms and rules of $ \cal{R}_{\sf S5} $ are given in the following:
 \vspace{0.5cm}	\\
\textbf{Initial sequents:}
\begin{center}
                \begin{tabular}{cccccc}
				\AXC{}
				\RL{\sf Ax}
				\UIC{$p,\Gamma\Rightarrow \Delta,p\,||\, {\cal H} $}
				\DP
				&
				&
				\AXC{}
				\RL{\sf L$\bot$}
				\UIC{$\bot,\Gamma\Rightarrow \Delta\,||\, {\cal H} $}
				\DP
				& 
				&
				\AXC{}
				\RL{\sf R$\top$}
				\UIC{$\Gamma\Rightarrow \Delta, \top\,||\, {\cal H} $}
				\DP
	\end{tabular}
\end{center}
\vspace{0.4cm}
\textbf{Propositional Rules:}
\begin{center}
\begin{tabular}{ccc}
				\AXC{$ \Gamma\Rightarrow  \Delta,A\,||\, {\cal H} $}
				\RL{\sf L$\neg $}
				\UIC{$ \neg A,\Gamma\Rightarrow  \Delta\,||\, {\cal H} $}
				\DP
				&
				&
				\AXC{$ A,\Gamma\Rightarrow  \Delta\,||\,{\cal H} $}
				\RL{\sf R$\neg $}
				\UIC{$ \Gamma\Rightarrow  \Delta,\neg A\,||\,{\cal H}$}
				\DP
				\\ [0.7cm]
				\AXC{$ A,\Gamma\Rightarrow  \Delta \,||\,{\cal H} $}
				\AXC{$  B,\Gamma\Rightarrow  \Delta \,||\,{\cal H}$}
				\RL{\sf L$\vee $}
				\BIC{$ A\vee B,\Gamma\Rightarrow  \Delta \,||\,{\cal H}$}
				\DP
				&
				&
				\AXC{$ \Gamma\Rightarrow  \Delta, A,B\,||\,{\cal H} $}
				\RL{\sf R$\vee $}
				\UIC{$\Gamma\Rightarrow  \Delta,A\vee B \,||\,{\cal H}$}
				\DP
				\\[0.7cm]
				\AXC{$ A,B,\Gamma\Rightarrow  \Delta \,||\,{\cal H}$}
				\RL{\sf L$\wedge $}
				\UIC{$A\wedge B,\Gamma\Rightarrow  \Delta \,||\,{\cal H}$}
				\DP
				&
				&
				\AXC{$ \Gamma\Rightarrow  \Delta,A \,||\,{\cal H}$}
				\AXC{$ \Gamma\Rightarrow  \Delta,B\,||\,{\cal H} $}
				\RL{\sf R$\wedge $}
				\BIC{$ \Gamma\Rightarrow  \Delta,A\wedge B\,||\,{\cal H} $}
				\DP
				\\[0.7cm]
				\AXC{$ \Gamma\Rightarrow  \Delta, A \,||\,{\cal H}$}
				\AXC{$ B,\Gamma\Rightarrow  \Delta \,||\,{\cal H}$}
				\RL{\sf L$\rightarrow $}
				\BIC{$ A\rightarrow B,\Gamma\Rightarrow  \Delta\,||\,{\cal H} $}
				\DP
				&
				&
				\AXC{$ A,\Gamma\Rightarrow  \Delta,B\,||\,{\cal H} $}
				\RL{\sf R$\rightarrow $}
				\UIC{$ \Gamma\Rightarrow  \Delta,A\rightarrow B\,||\,{\cal H} $}
				\DP
\end{tabular}
\end{center}
\vspace{0.3cm}
\textbf{Modal Rules:}
\begin{center}
\begin{tabular}{ccc}
				\AXC{$  A,M\Rightarrow N\,||\,{\cal H}\,|\,P\Rightarrow Q $}
				\RL{\sf L$\Diamond $}
				\UIC{$\Diamond A,M,P\Rightarrow Q,N\,||\,{\cal H} $}
				\DP
				&
				&
				\AXC{$\Gamma\Rightarrow  \Delta,\Diamond A,A \,||\,{\cal H}$}
				\RL{\sf R$\Diamond $}
				\UIC{$ \Gamma\Rightarrow  \Delta,\Diamond A\,||\,{\cal H}$}
				\DP
				\\[0.7cm]
				\AXC{$ A,\Box A,\Gamma\Rightarrow \Delta \,||\,{\cal H}$}
				\RL{\sf L$\Box $}
				\UIC{$\Box A,\Gamma\Rightarrow \Delta\,||\,{\cal H} $}
				\DP
				&
				&
				\AXC{$ M\Rightarrow N,A\,||\, {\cal H}\,|\,P\Rightarrow Q$}
				\RL{\sf R$\Box $}
				\UIC{$M,P\Rightarrow Q,N,\Box A \,||\,{\cal H}$}
				\DP
\end{tabular}
\end{center}
\vspace{0.3cm}
\textbf{Structural Rule:}
\begin{center}
\begin{tabular}{c}
					\AXC{$M,P_i\Rightarrow Q_i,N\,||\,{\cal H}\,| P\Rightarrow Q\,|{\cal G} $}
					\RL{\sf Exch}
					\UIC{$M,P\Rightarrow Q,N\,||\,{\cal H}\,| P_i\Rightarrow Q_i\,|{\cal G} $}
					\DP  	
			\end{tabular}
\end{center}
Let us make some remarks on the {\sf L$ \Diamond $},  {\sf R$ \Box $} and {\sf Exch} (abbreviates for exchange).
All formulas in the conclusions of these rules  are atomic or modal formulas. In a backward proof search by applying these rules, atomic formulas in the root  of the conclusions   move to the crown of  the premises as a new sequent. Suppose, $P$ in the antecedent  and $Q$ in the succedent of the root sequent of conclusion move to the crown of premise,
the formulas in $P$ and $Q$ are saved in the crowns, until  applications of the rule {\sf Exch} in a derivation. In other words,  The sequents in the crown  work as storage for  atomic formulas that    they might be  used  to get axioms. By applying the rule {\sf Exch}, the  multisets   $ P_i$ and $ Q_i $ come out from the crown  while $ P $ and $ Q $ move to the crown.

\begin{Examp}\label{exa}{\rm
The following sequents are derivable in $ \cal{R}_{\sf S5} $.
\begin{enumerate}
\item \label{exa4}
$\Rightarrow (r\wedge p)\rightarrow (q\rightarrow \Box(\Diamond (p\wedge q)\wedge \Diamond r)) $
\item \label{exa5}
$ \Rightarrow\Box(\Box\neg p\vee p)\rightarrow \Box(\neg p\vee \Box p) $
\end{enumerate}}
\end{Examp}
\begin{proof}
	$ \, $\\
\begin{prooftree}
	1.
	\AXC{}
	\RL{Ax}
	\UIC{$ p,q,r\Rightarrow \Diamond(p\wedge q),p $}
	\AXC{}
	\RL{Ax}
	\UIC{$ r,p,q\Rightarrow \Diamond(p\wedge q),q $}
	\RL{\sf R$\wedge $}
	\BIC{$ r,p,q\Rightarrow \Diamond(p\wedge q), p\wedge q $}
	\RL{\sf R$\Diamond $}
	\UIC{$ r,p,q\Rightarrow \Diamond(p\wedge q) $}
	\RL{\sf Exch}
	\UIC{$ \Rightarrow \Diamond(p\wedge q)\,||\,r,p,q\Rightarrow $}
	\AXC{}
	\RL{Ax}
	\UIC{$ r,p,q\Rightarrow \Diamond r,r $}
	\RL{\sf R$\Diamond $}
	\UIC{$ r,p,q\Rightarrow \Diamond r $}
	\RL{\sf Exch}
	\UIC{$ \Rightarrow \Diamond r \,||\,r,p,q\Rightarrow$}
	\RL{\sf R$\wedge $}
	\BIC{$ \Rightarrow\Diamond (p\wedge q)\wedge \Diamond r\,||\,r,p,q\Rightarrow $}
	\RL{\sf R$\Box $}
	\UIC{$ r,p,q\Rightarrow\Box(\Diamond (p\wedge q)\wedge \Diamond r) $}
	\RL{\sf L$\wedge $}
	\UIC{$ r\wedge p,q\Rightarrow  \Box(\Diamond (p\wedge q)\wedge \Diamond r) $}
	\RL{\sf R$\rightarrow $}
	\UIC{$ r\wedge p\Rightarrow q\rightarrow \Box(\Diamond (p\wedge q)\wedge \Diamond r) $}
	\RL{\sf R$\rightarrow $}
	\UIC{$\Rightarrow (r\wedge p)\rightarrow (q\rightarrow \Box(\Diamond (p\wedge q)\wedge \Diamond r)) $}
\end{prooftree}
\begin{prooftree}
	2.
	\AXC{}
	\RL{\sf Ax}
	\UIC{$\Box\neg p,\Box(\Box\neg p\vee p),p\Rightarrow p  \,||\,\Rightarrow p $}
	\RL{\sf L$\neg $}
	\UIC{$\neg p,\Box\neg p,\Box(\Box\neg p\vee p),p\Rightarrow  \,||\,\Rightarrow p $}
	\RL{\sf L$\Box $}
	\UIC{$ \Box\neg p,\Box(\Box\neg p\vee p),p\Rightarrow  \,||\,\Rightarrow p $}
	\RL{\sf Exch}
	\UIC{$ \Box\neg p,\Box(\Box\neg p\vee p)\Rightarrow p\,||\,p\Rightarrow $}
	\AXC{}
	\RL{\sf Ax}
	\UIC{$ p,\Box(\Box\neg p\vee p)\Rightarrow p\,||\,p\Rightarrow $}
	\RL{\sf L$\vee $}
	\BIC{$ \Box\neg p\vee p,\Box(\Box\neg p\vee p)\Rightarrow p \,||\,p\Rightarrow$}
	\RL{{\sf L$\Box$},}
	\UIC{$ \Box(\Box\neg p\vee p)\Rightarrow p\,||\,p\Rightarrow $}
	\RL{\sf R$\Box $}
	\UIC{$ \Box(\Box\neg p\vee p),p\Rightarrow  \Box p $}
	\RL{\sf R$\neg $}
	\UIC{$ \Box(\Box\neg p\vee p)\Rightarrow \neg p,\Box p $}
	\RL{\sf R$\vee $}
	\UIC{$ \Box(\Box\neg p\vee p)\Rightarrow \neg p\vee\Box p $}
	\RL{\sf R$\Box $}
	\UIC{$ \Box(\Box\neg p\vee p)\Rightarrow \Box(\neg p\vee\Box p) $}
	\RL{\sf R$\rightarrow $}
	\UIC{$\Rightarrow \Box(\Box\neg p\vee p)\rightarrow \Box(\neg p\vee\Box p) $}
\end{prooftree}
\end{proof}
\section{ Soundness}\label{section Soundness}
In this section we prove  soundness of the rules with respect to Kripke models.

A Kripke model $\mathcal{M}$ for S5 is a triple
$ \mathcal{M}=(W, R, V) $
where $ W $ is a set of states, $ R $ is an equivalence relation on $ W $ and
$ V: \varPhi\rightarrow \mathcal{P} (W) $ is a  valuation function, where
$ \varPhi $
is the set of propositional variables.
Suppose that $ w\in W $. We inductively
define the notion of a formula $ A $ being satisfied in $\mathcal{M}$ at state $ w $ as follows:
\begin{itemize}
	\item
	$\mathcal{M},w \vDash p \quad \text{iff} \quad  w\in V(p),\, \text{where}\,\,
	p\in \varPhi $,
	\item
	$\mathcal{M},w \vDash \neg A \quad \text{iff} \quad \mathcal{M},w \nvDash  A,$
	\item
	$\mathcal{M},w \vDash A\vee B  \quad \text{iff} \quad \mathcal{M},w \vDash  A \,\,\text{or}\,\, \mathcal{M},w \vDash  B,$
	\item
	$\mathcal{M},w \Rightarrow A\wedge B  \quad \text{iff} \quad \mathcal{M},w \vDash  A \,\,\text{and}\,\, \mathcal{M},w \vDash  B,$
	\item
	$\mathcal{M},w \vDash A\rightarrow B  \quad \text{iff} \quad \mathcal{M},w \nvDash  A \,\text{or}\, \mathcal{M},w \vDash  B,$
	\item
	$\mathcal{M},w \vDash \Diamond A \quad \text{iff} \quad  \,\,\mathcal{M},v \Rightarrow  A\,\,\text{for some}\,\, v\in W \,\,\text{such that}\,\, R(w,v), $
	\item
	$\mathcal{M},w \vDash \Box A \quad \text{iff} \quad \mathcal{M},v \Rightarrow  A \,\,\text{for all}\,\, v\in W \,\text{such that}\, R(w,v). $
\end{itemize}
We extend semantical notions to sequents in the following way:
\begin{itemize}
	\item 
	$ \mathcal{M},w \vDash \Gamma\Rightarrow \Delta\,||\,P_1\Rightarrow Q_1\,|\cdots P_n\Rightarrow Q_n $ iff 
	$ \mathcal{M},w \vDash \bigwedge\Gamma\rightarrow \bigvee\Delta\vee\bigvee_{i=1}^n\Box(\bigwedge P_i\rightarrow \bigvee Q_i)$.
	\item 
	$ \mathcal{M} \vDash  \Gamma\Rightarrow \Delta \,||\, {\cal H} $ iff $ \mathcal{M},w \Rightarrow    \Gamma\Rightarrow \Delta \,||\, {\cal H} $,  for all $  w $ in the domain of $ \mathcal{M} $.
	\item 
	$ \vDash  \Gamma\Rightarrow \Delta\,||\, {\cal H} $  iff
	$ \mathcal{M} \vDash  \Gamma\Rightarrow \Delta \,||\, {\cal H} $, for all S5  models $ \mathcal{M} $.
	\item 	
	The sequent $  \Gamma\Rightarrow \Delta\,||\, {\cal H} $ is called S5-valid if $ \vDash  \Gamma\Rightarrow \Delta \,||\, {\cal H}$.
\end{itemize}
\begin{Lem}\label{2.1}
	Let
	$\mathcal{M}=(W,R,V)$
	be a {\rm Kripke} model for
	$ {\sf S5} $.
	\begin{itemize}
		\item[{\rm (1)}]
		$\mathcal{M},w \Rightarrow \Box A \quad \text{iff} \quad \mathcal{M},w' \Rightarrow \Box A$, for all $ w'\in W $, where $ wRw' $.
		\item[{\rm (2)}]
		$\mathcal{M},w \Rightarrow \Diamond A \quad \text{iff} \quad \mathcal{M},w' \Rightarrow \Diamond A$, for all $ w'\in W $, where $ wRw' $.
		\item[{\rm (3)}]
		If
		$\mathcal{M},w \Rightarrow  A$, then  $\mathcal{M},w' \Rightarrow \Diamond A$, for all $ w'\in W $, where $ wRw' $.
	\end{itemize}
\end{Lem}
\begin{proof}
	The proof clearly follows from   the definition  of satisfiability  and the fact that 
	$ R $ is an equivalence relation.
\end{proof}
\begin{The}[Soundness]\label{Soundness}
	If
	$ \Gamma\Rightarrow \Delta \,||\, {\cal H}$
	is provable in $ \cal{R}_{\sf S5} $, then it is S5-valid.
\end{The}
\begin{proof}
	The proof is by induction on the height of the derivation of $ \Gamma\Rightarrow \Delta\,||\, {\cal H} $. Initial sequents are obviously valid
	in  every Kripke model for S5.  We only check the induction step for rules  {\sf R$ \Box $} and {\sf Exch}.  The rule {\sf L$ \Diamond $} can be verified similarly. 
	\begin{itemize}
		\item
	Rule {\sf R$ \Box $}:
	Suppose that the sequent $ \Gamma\Rightarrow \Delta \,||\, {\cal H} $ is $M,P\Rightarrow Q,N,\Box A \,||\,{\cal H} $, the conclusion of rule {\sf R$ \Box $}, with the premise $M\Rightarrow N,A\,||\,P\Rightarrow Q \,|\,{\cal H} $. For convenience,  let  the hypersequent $ {\cal H} $ be a sequent $ P_1\Rightarrow Q_1 $. Suppose, by induction hypothesis, that the premise is valid, i.e., for every Kripke model $ \cal M $ we have
	\begin{equation}\label{3}
	\text{If}\,\, \mathcal{M},w\vDash \bigwedge M,\,\, \text{then}\,\,   \mathcal{M},w \vDash\bigvee N\vee A \vee \Box(\bigwedge P\rightarrow \bigvee Q)\vee \Box(\bigwedge P_1\rightarrow \bigvee Q_1).
	\end{equation}
	Assume   the conclusion is not S5-valid i.e., there is a  model $ \mathcal{M}=(W,R,V) $ and $ w'\in W $  such that 
	\begin{align}
	&\mathcal{M},w'\vDash \bigwedge M\wedge \bigwedge P\label{1} \\
	& \mathcal{M},w'\nvDash    \bigvee Q\vee  \bigvee N\vee\Box A \vee\Box(\bigwedge P_1\rightarrow \bigvee Q_1).\label{2}
	\end{align}
Thus,
	$\mathcal{M},w'\nvDash \bigvee N\vee \Box(\bigwedge P\rightarrow \bigvee Q)\vee\Box(\bigwedge P_1\rightarrow \bigvee Q_1)$. Suppose $ w'Rw $, it follows from \ref{1}
and Lemma \ref{2.1} that $ \mathcal{M},w\vDash \bigwedge M $ and 
	$\mathcal{M},w\nvDash \bigvee N\vee \Box(\bigwedge P\rightarrow \bigvee Q)\vee\Box(\bigwedge P_1\rightarrow \bigvee Q_1)$. Therefore, by (\ref{3}) we have
  $ \mathcal{M},w\vDash A $, and so $ \mathcal{M},w'\vDash \Box A $. This leads to a contradiction with (\ref{2}).
	\item
	Rule {\sf Exch}:
		Suppose that the sequent $ \Gamma\Rightarrow \Delta \,||\, {\cal H} $ is  the conclusion of the rule {\sf Exch}.
		For convenience, let the hypersequent $ \cal{H} $ be a sequent $ P_1\Rightarrow Q_1 $ and let $ \cal{G} $ be a sequent $ P_2\Rightarrow Q_2 $. Suppose, by  induction hypothesis, that the premise is valid i.e., for every  model $ \cal M $  we have:
		\begin{align}\label{6}
		&\text{If}\,\, \mathcal{M},w \vDash \bigwedge M\wedge\bigwedge P_i,\,\, \text{then}\nonumber\\
		&    \mathcal{M},w \vDash \bigvee Q_i\vee \bigvee N\vee\Box(\bigwedge P_1\rightarrow \bigvee Q_1)\vee\Box(\bigwedge P\rightarrow \bigvee Q)\vee \Box(\bigwedge P_2\rightarrow \bigvee Q_2).
		\end{align}
	Assume   the conclusion is not valid i.e., there is a  model $ \mathcal{M}=(W,R,V) $ and $ w'\in W $  such that
	\begin{align}
	&\mathcal{M},w'\vDash \bigwedge M \wedge \bigwedge P\label{4} \\
	& \mathcal{M},w'\nvDash \bigvee Q\vee  \bigvee N\vee\Box(\bigwedge P_1\rightarrow \bigvee Q_1)\vee \Box(\bigwedge P_i\rightarrow \bigvee Q_i)\vee \Box(\bigwedge P_2\rightarrow \bigvee Q_2)  \label{5}
	\end{align}
	By \ref{5}, we have $\mathcal{M},w\nvDash \bigwedge P_i\rightarrow \bigvee Q_i $ for some $ w$ such that $w'Rw$. Thus,  $\mathcal{M},w\vDash\bigwedge P_i$, and $\mathcal{M},w\nvDash\bigvee Q_i$. In addition, because $M$ is a multiset of modal formula, using Lemma \ref{2.1} and (\ref{4}) we can conclude that $\mathcal{M},w\vDash \bigwedge M$. That means, $\mathcal{M},w\vDash \bigwedge M \wedge\bigwedge P_i$ and $\mathcal{M},w\nvDash\bigvee Q_i$ for some $w$ such that $w'Rw$. It follows from (\ref{6}) that 
	\[\mathcal{M},w \vDash  \bigvee N\vee\Box(\bigwedge P_1\rightarrow \bigvee Q_1)\vee\Box(\bigwedge P\rightarrow \bigvee Q)\vee \Box(\bigwedge P_2\rightarrow \bigvee Q_2).\]
This, Using Lemma \ref{2.1}, leads to a contradiction with (\ref{5}).
	\end{itemize}
\end{proof}
\section{Structural properties}\label{Structural properties}
In this section, we prove the admissibility of  weakening and  contraction rules, and also  some properties of  $ \cal{R}_{\sf S5} $, which are used to prove the admissibility of  cut rule. Some  (parts of)  proofs are omitted because they are easy or similar
to the proofs in \cite{Buss,negri2008structural,bpt}.

 The \textit{height} of a derivation is the greatest number of successive applications of rules
in it, where initial sequents have height $0$. The notation
$ \vdash_n\Gamma\Rightarrow \Delta\, ||\, \cal{H} $
means that 
$ \Gamma\Rightarrow \Delta\, ||\, \cal{H} $
is derivable  with a height of derivation
at most $ n $ in the system $ \cal{R}_{\sf S5}  $. A rule of $ \cal{R}_{\sf S5}  $ is said to be (height-preserving) \textit{admissible} if whenever an instance of its premise(s)  is (are)
derivable in $ \cal{R}_{\sf S5}  $ (with at most  height $ n $), then so is the corresponding instance of its conclusion. A rule of $ \cal{R}_{\sf S5}  $ is said to be  \textit{invertible} if whenever an instance of its conclusion is derivable in $ \cal{R}_{\sf S5}  $, then so is the corresponding instance of its premise(s).

The following lemma, shows that the propositional rules are height-preserving  invertible. The proof is by induction on the height of derivations.
\begin{Lem}\label{invertibility of propositional rules }
	All propositional rules are height-preserving invertible in $ \cal{R}_{\sf S5}  $.
\end{Lem}
In order to prove the admissibility of cut rule, we introduce the following structural rules, called {\sf Merge} and $  {\sf Merge}^{\sf c}  $. Rule $  {\sf Merge}^{\sf c}  $ allows us to merge two crown components in the crown, and Rule {\sf Merge} allows us to merge a component in the crown with the root component.
\begin{Lem}\label{Merge}
	The following rules are  height-preserving admissible
		\begin{center}
		\AXC{$ \Gamma\Rightarrow \Delta\,||\, {\cal H} \,|\,P_i\Rightarrow Q_i\,|\, {\cal G}\,|\,P_j\Rightarrow Q_j\,|\,{\cal I} $}
		\RL{$ {\sf  Merge}^{\sf c} $}
		\UIC{$\Gamma\Rightarrow \Delta\,||\, {\cal H} \,|\,P_i,P_j\Rightarrow Q_i,Q_j\,|\, {\cal G}\,|\,{\cal I}$}
		\DP
		$ \quad $
		\AXC{$ \Gamma\Rightarrow \Delta\,||\, {\cal H} \,|\,P_i\Rightarrow Q_i\,|\, {\cal G} $}
		\RL{\sf Merge}
		\UIC{$ \Gamma,P_i\Rightarrow Q_i, \Delta\,||\, {\cal H} \,|\, {\cal G} $}
		\DP
	\end{center}
\end{Lem}
\begin{proof}
	Both rules are proved simultaneously  by induction  on the height of the derivations of the premises. In both rules, if  the premise is an initial sequent, then  the conclusion is an initial sequent too. 
	For the induction step, we consider only   cases where  the last rule  is {\sf L$\Diamond $} or {\sf Exch}; since the rule {\sf R$\Box $} is treated symmetrically and  for the remaining 
	rules it  suffices to apply the induction hypothesis to the premise and
	then use the same rule to obtain deduction of the conclusion.\\
	For the rule  {\sf Merge} we consider the following  cases.
	
	Case 1. Let $ \Gamma=\Diamond A,M,P $ and 
	$ \Delta=Q,N $ 
	and let $ \Diamond A $ be the principal formula:
	\begin{prooftree}
		\AXC{$ \D $}
		\noLine
		\UIC{$ \vdash_n   A,M \Rightarrow N\,||\, {\cal H} \,|P_i\Rightarrow Q_i\,|\, {\cal G}\,|\, P\Rightarrow Q$}
		\RL{{\sf L$\Diamond $},}
		\UIC{$\vdash_{n+1} \Diamond A,M,P \Rightarrow Q,N\,||\, {\cal H} \,|P_i\Rightarrow Q_i\,|\, {\cal G} $}
	\end{prooftree}
	By induction hypothesis (IH), we have:
	\begin{prooftree}
		\AXC{$ \D $}
		\noLine
		\UIC{$ \vdash_n   A,M \Rightarrow N\,||\, {\cal H} \,|P_i\Rightarrow Q_i\,|\, {\cal G}\,|\, P\Rightarrow Q $}
		\RL{IH ($ {\sf  Merge}^{\sf c} $)}
		\UIC{$ \vdash_n   A,M \Rightarrow N\,||\, {\cal H} \,|\, {\cal G}\,|\, P,P_i\Rightarrow Q_i,Q$}
		\RL{{\sf L$\Diamond $}.}
		\UIC{$ \vdash_{n+1} \Diamond A,M,P,P_i\Rightarrow Q_i,Q,N \,||\,{\cal H} \,|\, {\cal G} $}
	\end{prooftree}

	Case 2. Let $ \Gamma=M,P$ and  $ \Delta=Q,N $, and let the last rule be as:
	\begin{prooftree}
		\AXC{$\vdash_n M,P_i\Rightarrow Q_i,N\,||\, {\cal H} \,|\,P\Rightarrow Q\,|\, {\cal G}   $}
		\RL{{\sf Exch}.}
		\UIC{$\vdash_{n+1} M,P\Rightarrow Q,N\,||\, {\cal H} \,|\,P_i\Rightarrow Q_i\,|\, {\cal G}  $}
	\end{prooftree}
	In this case, by induction hypothesis, the sequent $\vdash_n M,P,P_i\Rightarrow Q_i,Q,N\,||\, {\cal H} \,|\, {\cal G}   $ is obtained.\\
	For the rule $ {\sf  Merge}^{\sf c} $ we consider the following  cases.
	 
	Case 1. The  premise is derived by {\sf Exch}. If this rule does not apply to  $ P_i\Rightarrow Q_i $ and $ P_j\Rightarrow Q_j $, then the conclusion is obtained by applying induction hypothesis and then applying the same rule {\sf Exch}. Therefore, let the last rule be as follows in which the rule {\sf Exch} apply to the component $ P_i \Rightarrow Q_i $:
	\begin{prooftree}
		\AXC{$ \vdash_n M,P_i\Rightarrow Q_i,N\,||\, {\cal H} \,|\,P\Rightarrow Q\,|\, {\cal G}\,|\,P_j\Rightarrow Q_j\,|\,{\cal I}  $}
		\RL{{\sf Exch},}
		\UIC{$\vdash_{n+1} M,P\Rightarrow Q,N\,||\, {\cal H} \,|\,P_i\Rightarrow Q_i\,|\, {\cal G}\,|\,P_j\Rightarrow Q_j\,|\,{\cal I} $}
		\end{prooftree} 
where $ \Gamma=M,P $ and $ \Delta=Q,N $. Then the conclusion is derived as follows
		\begin{prooftree}
		\AXC{$ \vdash_n M,P_i\Rightarrow Q_i,N\,||\, {\cal H} \,|\,P\Rightarrow Q\,|\, {\cal G}\,|\,P_j\Rightarrow Q_j\,|\,{\cal I}   $}
		\RL{{\sf Merge} (IH)}
		\UIC{$\vdash_n M,P_i,P_j\Rightarrow Q_j,Q_i,N\,||\, {\cal H} \,|\,P\Rightarrow Q\,|\, {\cal G}\,|\,{\cal I} $}
		\RL{{\sf Exch}.}
		\UIC{$\vdash_{n+1} M,P\Rightarrow Q,N\,||\, {\cal H} \,|\,P_i,P_j\Rightarrow Q_j,Q_i\,|\, {\cal G}\,|\,{\cal I} $}	
	\end{prooftree}

Case 2. The  premise is derived by {\sf L$ \Diamond $}. 
	\begin{prooftree}
	\AXC{$ \vdash_n A,M\Rightarrow N\,||\, {\cal H} \,|\,P_i\Rightarrow Q_i\,|\, {\cal G}\,|\,P_j\Rightarrow Q_j\,|\,{\cal I}\,|\,P\Rightarrow Q$}
	\RL{{\sf L$ \Diamond $},}
	\UIC{$\vdash_{n+1} \Diamond A,M,P\Rightarrow Q,N\,||\, {\cal H} \,|\,P_i\Rightarrow Q_i\,|\, {\cal G}\,|\,P_j\Rightarrow Q_j\,|\,{\cal I}$}
\end{prooftree}
where $ \Gamma=\Diamond A,M,P $ and $ \Delta= Q,N $.
Then  the conclusion is obtained as follows:
\begin{prooftree}
	\AXC{$ \vdash_n A,M\Rightarrow N\,||\, {\cal H} \,|\,P_i\Rightarrow Q_i\,|\, {\cal G}\,|\,P_j\Rightarrow Q_j\,|\,{\cal I}\,|\,P\Rightarrow Q$}
	\RL{IH}
	\UIC{$\vdash_n A,M\Rightarrow N\,||\, {\cal H} \,|\,P_i,P_j\Rightarrow Q_i,Q_j\,|\, {\cal G}\,|\,{\cal I}\,|\,P\Rightarrow Q$}
	\RL{{\sf L$ \Diamond $}.}
	\UIC{$\vdash_{n+1} \Diamond A,M,P\Rightarrow Q,N\,|| {\cal H} \,|\,P_i,P_j\Rightarrow Q_i,Q_j\,|\, {\cal G}\,|\,{\cal I}$}
\end{prooftree}
\end{proof}
\subsection{Admissibility of  weakening }
 In this subsection, we prove the admissibility of  structural rules of external weakening $ {\sf EW} $, crown weakening  $ {\sf W}^{\sf c} $ , left  rooted weakening $ {\sf LW} $,  right rooted weakening $ {\sf RW} $.
 \begin{Lem}\label{external weakening}
 	The rule of external weakening:
 	\begin{prooftree}
 		\AXC{$ \Gamma\Rightarrow \Delta\,||\, {\cal H}\,|\, \cal{G} $}
 		\RL{$ \sf EW $}
 		\UIC{$ \Gamma\Rightarrow \Delta\,||\, {\cal H}\,|\, P\Rightarrow Q\,|\, \cal{G} $}
 	\end{prooftree}
 	is height-preserving admissible in $ {\cal R}_{\sf S5} $.
 \end{Lem}
 \begin{proof}
 	 By straightforward induction on the height of the derivation of the premise.
 \end{proof}
 \begin{Lem}\label{modal and atomic weakening}
 The  rule of crown weakening: 
 \begin{prooftree}
 	\AXC{$ \Gamma\Rightarrow\Delta\,||\, {\cal H}\,|\,P_i\Rightarrow Q_i\,|\, {\cal G} $}
 	\RL{$ {\sf W}^{\sf c} $}
 	\UIC{$ \Gamma\Rightarrow\Delta\,||\, {\cal H}\,|\,P,P_i\Rightarrow Q_i,Q\,|\, {\cal G} $}
 	\end{prooftree}
is height-preserving admissible in $ {\cal R}_{\sf S5} $.
 \end{Lem}
\begin{proof}
	{\rm 
	It follows by the height-preserving admissibility of the two rules $ \sf EW $  and $ {\sf  Merge}^{\sf c} $.
}
\end{proof}
\begin{Lem} \label{W2}
	The rules of left and right weakening:
	\begin{center}
		\AXC{$ \Gamma\Rightarrow \Delta\,||\, {\cal H} $}
		\RL{\sf LW}
		\UIC{$ A,\Gamma\Rightarrow \Delta\,||\, {\cal H} $}
		\DP
		$\qquad$
		\AXC{$ \Gamma\Rightarrow \Delta \,||\, {\cal H}$}
		\RL{\sf RW,}
		\UIC{$ \Gamma\Rightarrow \Delta,A\,||\, {\cal H} $}
		\DP
	\end{center}
	are  admissible, where $ A $   is an arbitrary formula.
\end{Lem}
\begin{proof}
	{\rm
		Both rules are proved simultaneously  by induction on the complexity of $ A $ with subinduction on the height of the derivations. The admissibility of the rules for atomic formula follows by admissibility of the two rules {\sf EW} and {\sf Merge}. The admissibility of the rules for modal  formula are straightforward by induction on the height of the derivation. The other cases are proved easily.
	}
\end{proof}
\subsection{Invertibility }
In this subsection, first we introduce a normal form called Quasi Normal Form, which are used to prove the  admissibility of the contraction and cut rules. Then we show that the structural and modal rules  are invertible.
\begin{Def}[Quasi-literal]
	A quasi-literal is an atomic formula, a negation of atomic formula,  a modal formula, or a negation of modal formula.
	\end{Def}
\begin{Def}[Quasi-clause and quasi-phrase]
A quasi-clause (quasi-phrase) is either a single quasi-literal or a conjunction (disjunction) of quasi-literals.
	\end{Def}
\begin{Def}
	A formula is in  conjunctive (disjunctive) quasi-normal form, abbreviated CQNF (DQNF), iff it is a conjunction (disjunction) of quasi-clauses.
	\end{Def}
\begin{Examp}
	The formula $(\neg\Box(A\rightarrow B)\vee p \vee \Diamond C)\wedge(\neg q)\wedge(\Box\Diamond A\vee \neg \Diamond(A\wedge B)\vee\neg r)$ is in CQNF, which $(\neg\Box(A\rightarrow B)\vee p \vee \Diamond C)$, $(\neg q)$, and $(\Box\Diamond A\vee \neg \Diamond(A\wedge B)\vee\neg r)$ are quasi-clauses.
\end{Examp}
Without loss of generality we can assume that every formula in CQNF is as follows
\[\bigwedge_{i=1}^k(\bigvee P_i\vee\bigvee \neg Q_i\vee\bigvee M_i\vee\bigvee \neg N_i),\]
and every formula in DQNF is as follows
\[\bigvee_{i=1}^l(\bigwedge P_i\vee\bigwedge \neg Q_i\vee\bigwedge M_i\vee\bigwedge \neg N_i),\]
where $P_i$ and $Q_i$ are multisets of atomic formulae and $M_i$ and $N_i$ are multisets of modal formulae.

Similar to the  propositional logic, Every formula  can be transformed into an equivalent formula in CQNF, and an equivalent formula in DQNF.

For these quasi-normal forms we have the following lemmas.
\begin{Lem}\label{NFcut1}
	Let $ A $ be an arbitrary  formula, and let $\bigwedge_{i=1}^k(\bigvee \neg P_i\vee\bigvee  Q_i\vee\bigvee \neg M_i\vee\bigvee  N_i)$ and $\bigvee_{i=1}^{l}(\bigwedge P'_i\vee\bigwedge \neg Q'_i\vee\bigwedge M'_i\vee\bigwedge \neg N'_i)$ be  equivalent formulae in CQNF and DQNF, respectively, for $A$. Then
	\begin{itemize}
		\item[{\rm (i)}]
		$ \vdash\Gamma\Rightarrow \Delta, A\,||\, {\cal H} $
		iff
		$\vdash M_i,P_i,\Gamma\Rightarrow \Delta,   Q_i, N_i\,||\, {\cal H}$,
		for every $i=1,\ldots,k$.
		\item[{\rm (ii)}]
		$\vdash A,\Gamma\Rightarrow \Delta\,||\, {\cal H} $
		iff  
		$\vdash	M'_i,  P'_i,\Gamma\Rightarrow \Delta,Q'_i, N'_i\,||\, {\cal H}$,
		for every $i=1,\ldots,l$.	
	\end{itemize}
\end{Lem}
\begin{proof}
	We consider the first assertion; the second is treated symmetrically.\\
 $ (\Leftarrow) $. Suppose that $ M_i,P_i,\Gamma\Rightarrow \Delta,   Q_i, N_i\,||\, {\cal H}$ is derivable
 for every $i=1,\ldots,k$. By applying propositional rules, where non-modal subformulae of $ A $  are principal formulae in a backward proof search with  $ \Gamma\Rightarrow \Delta, A\,||\, {\cal H} $ in the bottom.
 We obtain a derivation, in which the topsequents are  $ M_i,P_i,\Gamma\Rightarrow \Delta,   Q_i, N_i\,||\, {\cal H}$, $i=1,\ldots,k$.\\
$ (\Rightarrow) $. Suppose 	$ \Gamma\Rightarrow \Delta, A \,||\, {\cal H}$ is derivable. By applying the invertibility of the propositional rules, we reach a derivation for $ M_i,P_i,\Gamma\Rightarrow \Delta,   Q_i, N_i\,||\, {\cal H}$,
for every $i=1,\ldots,k$.
\end{proof}
\begin{Cor}\label{left to right1}
	Both rules in the above lemma from left to right are height-preserving admissible. In other words:
		\item[{\rm (i)}]
	If $ \vdash_n\Gamma\Rightarrow \Delta, A\,||\, {\cal H} $,
	then
	$\vdash_n M_i,P_i,\Gamma\Rightarrow \Delta,   Q_i, N_i\,||\, {\cal H}$,
	for every $i=1,\ldots,k$.
	\item[{\rm (ii)}]
   If	$\vdash_n A,\Gamma\Rightarrow \Delta\,||\, {\cal H} $,
	then 
	$\vdash_n	M'_i,  P'_i,\Gamma\Rightarrow \Delta,Q'_i, N'_i\,||\, {\cal H}$,
	for every $i=1,\ldots,l$.
\end{Cor}
\begin{Lem}\label{NFcut2}
		Let $ A $ be an arbitrary  formula, and let $\bigwedge_{i=1}^k(\bigvee\neg P_i\vee\bigvee  Q_i\vee\bigvee\neg  M_i\vee\bigvee N_i)$ and $\bigvee_{i=1}^{l}(\bigwedge P'_i\vee\bigwedge \neg Q'_i\vee\bigwedge M'_i\vee\bigwedge \neg N'_i)$ be  equivalent formulae in CQNF and DQNF, respectively, for $A$. Then 
	\begin{itemize}
		\item[{\rm (i)}]
		$\vdash \Gamma\Rightarrow \Delta,\Box A\,||\, {\cal H} $
		iff
		$\vdash	M_i,\Gamma\Rightarrow \Delta,  N_i\,||\, {\cal H}\,|\, P_i\Rightarrow Q_i$,
		for every $i=1,\ldots,k$.
		\item[{\rm (ii)}]
		$ \vdash\Diamond A,\Gamma\Rightarrow \Delta\,||\, {\cal H} $
		iff 
		$\vdash	M'_i,  \Gamma\Rightarrow \Delta, N'_i\,||\, {\cal H}\,|\,P'_i\Rightarrow Q'_i$,
	for every $i=1,\ldots,l$.
	\end{itemize}
\end{Lem}
\begin{proof}
	We prove the first assertion; the second is proved by a similar argument. For the direction from right to left observe that by Lemma \ref{NFcut1},  $ \vdash	M_i,\Gamma\Rightarrow \Delta,  N_i\,||\, {\cal H}\,|\,P_i\Rightarrow Q_i $ implies 
	$  \vdash	M_i,M,P\Rightarrow N,Q,  N_i\,||\, {\cal H}\,|\, P_i\Rightarrow Q_i$, where $ M,P$ and $N,Q $ are  multisets of
	atomic and modal formulae   corresponding to  formulae in $ \Gamma $ and $ \Delta$. Then by applying rule {\sf Exch}  we have:
	\begin{prooftree}
		\AXC{$ \vdash	M_i,M,P\Rightarrow N,Q,  N_i\,||\, {\cal H}\,|\,P_i\Rightarrow Q_i$}
		\RL{{\sf Exch}.}
		\UIC{$ \vdash	M_i,M,P_i\Rightarrow N,Q_i,  N_i\,||\, {\cal H} \,|\,P\Rightarrow Q$}
	\end{prooftree}
Therefore, $ \vdash	M_i,M,P_i\Rightarrow N,Q_i,  N_i\,||\, {\cal H}\,|\,P\Rightarrow Q $, for $ i=1,\ldots,k $, and then  by Lemma \ref{NFcut1} we have\\
 $\vdash M\Rightarrow N,A\,||\, {\cal H}\,|\,P\Rightarrow Q $. Hence, by applying rule {\sf R$ \Box $} we have:
\begin{prooftree}
	\AXC{$\vdash M\Rightarrow N,A\,||\, {\cal H}\,|\,P\Rightarrow Q$}
	\RL{\sf R$ \Box $}
	\UIC{$\vdash M,P\Rightarrow N,Q,\Box A\,||\, {\cal H} $}
	\end{prooftree}
Again by applying Lemma \ref{NFcut1} we have $\vdash \Gamma\Rightarrow  \Delta,\Box A \,||\, {\cal H}$.\\
	For the opposite direction, by induction on $ n $, we prove that if $ \vdash_n\Gamma\Rightarrow  \Delta,\Box A\,||\, {\cal H}  $, then\\
	 $\vdash_n	M_i,\Gamma\Rightarrow \Delta,  N_i\,||\, {\cal H}\,|\,P_i\Rightarrow Q_i$, for $ i=1,\ldots,k $. For  $ n=0 $, if $ \Gamma\Rightarrow  \Delta,\Box A\,||\, {\cal H}  $ is an initial sequent, then  $ \Box A  $ is not principal, and so $	M_i,\Gamma\Rightarrow \Delta,  N_i\,||\, {\cal H}\,|\,P_i\Rightarrow Q_i$ is an initial sequent too, for $  i=1,\ldots,k $. Assume height-preserving admissible up to height n, and let
	  $$ \vdash_{n+1} \Gamma\Rightarrow  \Delta,\Box A\,||\, {\cal H}   $$ If $ \Box A $ is not principal, we apply induction hypothesis to the premise(s) and then use the same rule to obtain deductions of $ 		M_i,\Gamma\Rightarrow \Delta,  N_i\,||\,{\cal H} \,|\, P_i\Rightarrow Q_i$, for $  i=1,\ldots,k $.  If on the other hand $ \Box A $ is principal, the derivation ends with
	\begin{prooftree}
		\AXC{$\vdash_n M\Rightarrow  N, A\,||\,{\cal H}\,|\,P\Rightarrow Q $}
		\RL{{\sf R$ \Box $},}
		\UIC{$\vdash_{n+1}M,P\Rightarrow  Q,N,\Box A\,||\,{\cal H} $}
	\end{prooftree}
	where $ \Gamma= M,P $ and $ \Delta= Q,N$. Therefore, by applying Corollary \ref{left to right1}, we have 
	$$\vdash_n M_i,P_i,M\Rightarrow  N, N_i,Q_i\,||\,{\cal H}\,|\,P\Rightarrow Q, $$
	for $  i=1,\ldots,k $. Hence, by applying rule {\sf Exch}  the conclusion is obtained as follows:
		\begin{prooftree}
		\AXC{$\vdash_n M_i,P_i,M\Rightarrow  N, N_i,Q_i\,||\,{\cal H} \,|\,P\Rightarrow Q$ }
		\RL{{\sf Exch}.}
		\UIC{$\vdash_{n+1} M_i,P,M\Rightarrow  N, N_i,Q\,||\,{\cal H}\,|\,P_i\Rightarrow Q_i $ }
	\end{prooftree}
\end{proof}
\begin{Cor}\label{left to right2}
	Both rules in the above lemma from left to right are  height-preserving admissible. In other words:
	\begin{itemize}
		\item[{\rm (i)}]
		If $\vdash_n \Gamma\Rightarrow \Delta,\Box A\,||\, {\cal H} $,
		then
		$\vdash_n	M_i,\Gamma\Rightarrow \Delta,  N_i\,||\, {\cal H}\,|\,P_i\Rightarrow Q_i$,
		for every $i=1,\ldots,k$.
		\item[{\rm (ii)}]
		If $ \vdash_n\Diamond A,\Gamma\Rightarrow \Delta\,||\, {\cal H} $,
		then 
		$\vdash_n	M'_i,  \Gamma\Rightarrow \Delta, N'_i\,||\, {\cal H}\,|\,P'_i\Rightarrow Q'_i$,
		for every $i=1,\ldots,l$.
	\end{itemize}
\end{Cor}
\begin{Lem}	\label{inversion for modal rules}
	The structural and modal rules are invertible in $ \cal{R}_{\sf S5}  $.
\end{Lem}
\begin{proof}
	Rules {\sf L$ \Box $}	and {\sf R$ \Diamond $} have repetition of the principle formula in the premises; so that we can obtain a derivation of premises of these rules by weakening their conclusions.
	
 We prove that the rules  {\sf R$ \Box $} and {\sf Exch} are invertible; the rule {\sf L$ \Diamond $} is treated similarly.\\
  Let  $\vdash  M,P\Rightarrow Q,N,\Box A\,||\, {\cal H}$, and let $\bigwedge_{i=1}^n(\bigvee P_i\vee\bigvee \neg Q_i\vee\bigvee M_i\vee\bigvee \neg N_i)$ be an equivalent formula in CQNF for $ A $. Then, by applying Lemma \ref{NFcut2} we have
   $$\vdash M_i,M,P\Rightarrow Q, N,N_i\,||\, {\cal H}\,|\,P_i\Rightarrow Q_i,  $$ for $ i=1,\ldots,n $. Thus, by applying rule {\sf Exch} we have
	\begin{prooftree}
		\AXC{$M_i,M,P\Rightarrow Q, N,N_i\,||\, {\cal H}\,|\,P_i\Rightarrow Q_i $}
		\RL{\sf Exch}
		\UIC{$M_i,M,P_i\Rightarrow Q_i, N,N_i\,||\, {\cal H}\,|\,P\Rightarrow Q $}
		\end{prooftree}
	Therefore, $ \vdash M_i,M,P_i\Rightarrow Q_i, N,N_i\,||\, {\cal H}\,|\, P\Rightarrow Q$, for $ i=1,\ldots,n $, and then again by applying  Lemma \ref{NFcut1} we have $\vdash  M\Rightarrow  N,A\,||\, {\cal H}\,|\, P\Rightarrow Q  $.
	
	Finally, we prove that the rule {\sf Exch}
	is invertible. Let $ \D $ be a derivation  of 
	$$ M,P\Rightarrow Q,N\,||\,{\cal H}\,|\, P_i\Rightarrow Q_i\,|\,{\cal G}. $$ If  the last rule applied in $ \D $ is propositional  rules or {\sf Exch}, where the rule {\sf Exch} does not apply to $ P_i \Rightarrow Q_i  $, apply induction hypothesis to the premise and then apply the last rule with the same principal formula to obtain deduction of $M,P_i\Rightarrow Q_i,N\,||\,{\cal H}\,|\, P\Rightarrow Q\,|\,{\cal G} $. Therefore, let the last rule be {\sf R$ \Box $} as follows:
		\begin{prooftree}
		\AXC{$M\Rightarrow N', A\,||\,{\cal H}\,|\, P_i\Rightarrow Q_i\,|\,{\cal G}\,|\,P\Rightarrow Q $}
		\RL{{\sf R$ \Box $},}
		\UIC{$ M,P\Rightarrow Q,N',\Box A\,||\,{\cal H}\,|\, P_i\Rightarrow Q_i\,|\,{\cal G} $}
	\end{prooftree}
	where $ N=N',\Box A $,  we have
	\begin{prooftree}
		\AXC{$M\Rightarrow N', A\,||\,{\cal H}\,|\, P_i\Rightarrow Q_i\,|\,{\cal G}\,|\,P\Rightarrow Q $}
		\RL{{\sf R$ \Box $}.}
		\UIC{$ M,P_i\Rightarrow Q_i,N',\Box A\,||\,{\cal H}\,|\,{\cal G}\,|\,P\Rightarrow Q $ }
	\end{prooftree}
If the last rule is  {\sf L$ \Diamond $},  the proof is similarly. If the last rule is {\sf Exch} which is applied to $ P_i \Rightarrow Q_i  $,   the conclusion is obtained.
	\end{proof}
\subsection{Admissibility of  contraction }
In order to prove the admissibility of the contraction  rule, we need  the following lemma.
\begin{Lem}\label{*}
The following rules are admissible.
	\begin{center}
		\AXC{$ \Diamond A,\Gamma\Rightarrow \Delta\,||\,{\cal H} $}
		\LL{{\rm (\rom{1})}}
		\UIC{$ A,\Gamma\Rightarrow  \Delta\,||\,{\cal H} $}
		\DP
		$ \qquad $
		\AXC{$ \Gamma\Rightarrow \Delta,\Box A\,||\,{\cal H} $}
		\LL{\rm (\rom{2})}
		\UIC{$\Gamma\Rightarrow \Delta, A \,||\,{\cal H}$}
		\DP
	\end{center}
\end{Lem}
\begin{proof}
	We only consider the part $ {\rm (\rom{1})} $; the other is proved similarly. Let $ \vdash \Diamond A,\Gamma\Rightarrow \Delta\,||\,{\cal H} $,  and  let $\bigvee_{i=1}^{k}(\bigwedge P'_i\vee\bigwedge \neg Q'_i\vee\bigwedge M'_i\vee\bigwedge \neg N'_i)$ be an equivalent formula in DQNF for $ A $. Then we have  
\begin{prooftree}
	\AXC{$\vdash \Diamond A,\Gamma\Rightarrow \Delta\,||\,{\cal H} $}
	\RL{Lemma \ref{NFcut2}}
	\UIC{$\vdash M'_i,\Gamma\Rightarrow \Delta, N'_i\,||\,{\cal H}\,|\, P'_i\Rightarrow Q'_i$}
	\RL{\sf Merge}
	\UIC{$\vdash M'_i,\Gamma,P'_i\Rightarrow Q'_i,\Delta, N'_i \,||\,{\cal H}$}
	\RL{Lemma  \ref{NFcut1}}
	\UIC{$ \vdash A,\Gamma\Rightarrow \Delta \,||\,{\cal H}$}
	\end{prooftree}	
\end{proof}
\begin{Cor}\label{**}
	The following rule is  admissible.
	\begin{prooftree}
		\AXC{$ \Diamond A,\Gamma\Rightarrow \Delta,\Box B\,||\,{\cal H} $}
		\UIC{$ A,\Gamma\Rightarrow \Delta,B\,||\,{\cal H} $}
	\end{prooftree}
\end{Cor}
 We now  prove   the admissibility of  rules left rooted contraction {\sf LC}, right rooted contraction {\sf RC}, left crown contraction $ {\sf LC}^{\sf c} $, right crown contraction $ {\sf RC}^{\sf c} $ and finally, external contraction {\sf EC},  which are required for the proof of the  admissibility of  cut rule.  First we consider the rules {\sf LC} and {\sf RC} for atomic formula, these rules and rules $ {\sf LC}^{\sf c} $ and $ {\sf RC}^{\sf c} $  are proved simultaneously.
\begin{Lem} \label{atomic contraction}
	The following rules are height-preserving admissible, 
	\[\begin{array}{cc}
	\AXC{$ p,p,\Gamma\Rightarrow \Delta\,||\,{\cal H}  $}
	\RL{\sf LC}
	\UIC{$ p,\Gamma\Rightarrow \Delta\,||\,{\cal H} $}
	\DP
	&
	\AXC{$ \Gamma\Rightarrow \Delta,q,q \,||\,{\cal H} $}
	\RL{\sf RC,}
	\UIC{$ \Gamma\Rightarrow \Delta,q\,||\,{\cal H} $}
	\DP
		\\[0.5cm]		
\AXC{$ \Gamma\Rightarrow\Delta\,||\,{\cal H}\,|\, p,p,P_i\Rightarrow Q_i\,|\,{\cal G} $}
\RL{$ {\sf LC} ^{\sf c}$}
\UIC{$ \Gamma\Rightarrow\Delta\,||\,{\cal H}\,|\, p,P_i\Rightarrow Q_i\,|\,{\cal G} $}
\DP
&
\AXC{$ \Gamma\Rightarrow\Delta\,||\,{\cal H}\,|\, P_i\Rightarrow Q_i,q,q\,|\,{\cal G} $}
\RL{$ {\sf RC}^{\sf c} $}
\UIC{$ \Gamma\Rightarrow\Delta\,||\,{\cal H}\,|\, P_i\Rightarrow Q_i,q\,|\,{\cal G} $}
\DP
\end{array}	\]
where $ p $ and $ q $ are atomic formulae.
\end{Lem}
\begin{proof}
	{\rm
		All rules are    proved simultaneously by induction on  the height of  derivation of the premises. In all cases, if the premise is an initial sequent, then the conclusion is an initial sequent too. We  consider some cases; the other cases are proved by a similar argument.
		
For  the rule $ {\sf LC} ^{\sf c}$, let $ \Gamma=M,P$ and $ \Delta=Q,N$,  and let  the last rule  be 
\begin{prooftree}
	\AXC{$ \D $}
	\noLine
	\UIC{$\vdash_n M,p,p,P_i\Rightarrow Q_i,N\,||\,{\cal H}\,|\, P\Rightarrow Q\,|\,{\cal G} $}
	\RL{{\sf Exch},}
	\UIC{$\vdash_{n+1} M,P\Rightarrow Q,N\,||\,{\cal H}\,|\, p,p,P_i\Rightarrow Q_i\,|\,{\cal G} $}
\end{prooftree}
 then we have
\begin{prooftree}
	\AXC{$ \D $}
	\noLine
	\UIC{$ \vdash_n M,p,p,P_i\Rightarrow Q_i,N\,||\,{\cal H}\,|\, P\Rightarrow Q\,|\,{\cal G} $}
	\RL{IH ({\sf LC})}
	\UIC{$\vdash_n M,p,P_i\Rightarrow Q_i,N\,||\,{\cal H}\,|\, P\Rightarrow Q\,|\,{\cal G}  $}
	\RL{{\sf Exch}.}
	\UIC{$\vdash_{n+1}M,P\Rightarrow Q,N\,||\,{\cal H}\,|\, p,P_i\Rightarrow Q_i\,|\,{\cal G} $}
\end{prooftree} 
For the other last rules $\sf  R $, use induction hypothesis  on the premise, and then apply the rule $ \sf R $.\\

For  the rule $\sf LC $, let  $ \Gamma=M,P$ and $ \Delta=Q,N$, and let the last rule be as follows
\begin{prooftree}
	\AXC{$ \D $}
	\noLine
	\UIC{$\vdash_n M,P_i\Rightarrow Q_i,N\,||\, {\cal H}\,|\,p,p,P \Rightarrow Q\,|\, {\cal G}$}
	\RL{{\sf Exch},}
	\UIC{$\vdash_{n+1} p,p,M,P\Rightarrow Q,N\,||\, {\cal H}\,|\, P_i\Rightarrow Q_i\,|\, {\cal G} $}
\end{prooftree}
then we have
\begin{prooftree}
	\AXC{$ \D $}
	\noLine
	\UIC{$\vdash_n M,P_i\Rightarrow Q_i,N\,||\, {\cal H}\,|\,p,p,P \Rightarrow Q\,|\, {\cal G}$}
	\RL{IH ($ {\sf LC} ^{\sf c}$)}
	\UIC{$\vdash_n M,P_i\Rightarrow Q_i,N\,||\, {\cal H}\,|\,p,P \Rightarrow Q\,|\, {\cal G}$}
	\RL{{\sf Exch}.}
	\UIC{$\vdash_{n+1} M,p,P\Rightarrow Q,N\,||\, {\cal H}\,|\,P_i \Rightarrow Q_i\,|\, {\cal G} $}
\end{prooftree}
Let the last rule be {\sf L$\Diamond $} as:
\begin{prooftree}
	\AXC{$ \D $}
	\noLine
	\UIC{$\vdash_n  B,M\Rightarrow N\,||\, {\cal H}\,|\,p,p,P\Rightarrow Q$}
	\RL{{\sf L$\Diamond $},}
	\UIC{$\vdash_{n+1} p,p,\Diamond B,M,P\Rightarrow Q,N\,||\, {\cal H} $}
\end{prooftree}
where $ \Gamma=\Diamond B,M,P$ and $ \Delta=Q,N$. Then we have
\begin{prooftree}
	\AXC{$ \D $}
	\noLine
	\UIC{$\vdash_n  B,M\Rightarrow N\,||\, {\cal H}\,|\,p,p,P\Rightarrow Q $}
	\RL{IH ($ {\sf LC} ^{\sf c}$)}
	\UIC{$\vdash_n B,M\Rightarrow N\,||\,{\cal H}\,|\,p,P\Rightarrow Q  $}
	\RL{{\sf L$\Diamond $},}
	\UIC{$\vdash_{n+1} p,P,\Diamond B,M\Rightarrow N,Q\,||\, {\cal H}  $}
\end{prooftree}
The rule {\sf R$ \Box $} is treated similarly; for the other rules use induction hypothesis and then use the same rule.
\begin{Lem}
	The rules of contraction, 
		\begin{center}
		\AXC{$ A,A,\Gamma\Rightarrow \Delta\,||\, {\cal H} $}
		\RL{\sf LC}
		\UIC{$ A,\Gamma\Rightarrow \Delta\,||\, {\cal H} $}
		\DP
		$\qquad$
		\AXC{$ \Gamma\Rightarrow \Delta,A,A\,||\, {\cal H} $}
		\RL{{\sf RC},}
		\UIC{$ \Gamma\Rightarrow \Delta,A\,||\, {\cal H} $}
		\DP
	\end{center}
are  admissible.
\end{Lem}
\begin{proof}
		Both rules are proved simultaneously  by induction on the complexity of $ A $ with subinduction on the height of the derivations. The lemma  holds for atomic  formula $ A $, by Lemma \ref{atomic contraction}. For the other cases, if $ A $ is not principal in the last rule (either modal or propositional), apply inductive
		hypothesis to the premises and then apply the last rule. Here we only consider the rule $ {\sf LC} $; the admissibility of the rule {\sf RC} is proved similarly.  Suppose   
		$ \vdash A,A,\Gamma\Rightarrow \Delta\,||\, {\cal H} $, we consider some cases, in which $ A $ is principal formula in the last rule:
	\end{proof}

Case 1. $ A= \Diamond B $:
		\begin{prooftree}
			\AXC{$ \D $}
			\noLine
			\UIC{$\vdash B,\Diamond B,M\Rightarrow N\,||\, {\cal H}\,|\,P\Rightarrow Q $}
			\RL{{\sf L$\Diamond $},}
			\UIC{$\vdash  \Diamond B,\Diamond B,M,P\Rightarrow N,Q\,||\, {\cal H} $}
		\end{prooftree}
		 where $ \Gamma=M,P$ and $ \Delta=N,Q$. Then we have
		\begin{prooftree}
			\AXC{$ \D $}
			\noLine
			\UIC{$\vdash B,\Diamond B,M\Rightarrow N\,||\, {\cal H}\,|\,P\Rightarrow Q $}
			\RL{ Lemma \ref{*} }
			\UIC{$\vdash B, B,M\Rightarrow N\,||\, {\cal H}\,|\,P\Rightarrow Q $}
			\RL{ IH }
			\UIC{$\vdash  B,M\Rightarrow N\,||\, {\cal H}\,|\,P\Rightarrow Q$}
			\RL{{\sf L$\Diamond $}.}
		\UIC{$\vdash  \Diamond B,M,P\Rightarrow N,Q\,||\, {\cal H} $}
		\end{prooftree}
	
Case 2.  $ A=\Box B$:
\begin{prooftree}
	\AXC{$ \D $}
	\noLine
	\UIC{$\vdash   B,\Box B,\Box B,\Gamma\Rightarrow \Delta\,||\, {\cal H} $}
	\RL{{\sf L$\Box $},}
	\UIC{$\vdash \Box B,\Box B,\Gamma\Rightarrow \Delta\,||\, {\cal H} $}
\end{prooftree}
then
\begin{prooftree}
	\AXC{$ \D $}
	\noLine
	\UIC{$\vdash B,\Box B,\Box B,\Gamma\Rightarrow \Delta\,||\, {\cal H} $}
	\RL{IH}
	\UIC{$\vdash  B,\Box B,\Gamma\Rightarrow \Delta\,||\, {\cal H} $}
	\RL{{\sf L$\Box $}.}
	\UIC{$\vdash \Box B,\Gamma\Rightarrow \Delta\,||\, {\cal H} $}
\end{prooftree}

Case 3. Let $ A=B\rightarrow C $ and let the last rule as	
\begin{prooftree}
	\AXC{$ \D_1 $}
	\noLine
	\UIC{$\vdash  B\rightarrow C,\Gamma\Rightarrow \Delta,B\,||\, {\cal H} $}
	\AXC{$ \D_2 $}
	\noLine
	\UIC{$\vdash  C,B\rightarrow C,\Gamma\Rightarrow \Delta\,||\, {\cal H} $}
	\RL{{\sf L$\rightarrow $},}
	\BIC{$\vdash  B\rightarrow C,B\rightarrow C,\Gamma\Rightarrow \Delta\,||\, {\cal H} $}
\end{prooftree}		
By applying  the invertibility of the rule {\sf L$ \rightarrow $}, see Lemma \ref{invertibility of propositional rules },  to the first premise, we get 
$$ \vdash  \Gamma\Rightarrow \Delta,B,B\,||\, {\cal H} $$ and
applying to the second premise, we get $\vdash  C,C,\Gamma\Rightarrow \Delta\,||\, {\cal H} $. We then use the induction hypothesis and obtain $\vdash  \Gamma\Rightarrow \Delta,B\,||\, {\cal H} $  and $\vdash  C,\Gamma\Rightarrow \Delta\,||\, {\cal H} $. Thus by {\sf L$\rightarrow $}, we get $\vdash  B\rightarrow C,\Gamma\Rightarrow \Delta \,||\, {\cal H}$.\\
	}
\end{proof}
\begin{Lem}
	The rule of external  contraction,
		\begin{center}
		\AXC{$  \Gamma\Rightarrow \Delta\,||\,{\cal H}\,|\, P_i\Rightarrow Q_i\,|\, P_i\Rightarrow Q_i\,|\, {\cal G}  $}
		\RL{$ \sf EC $}
		\UIC{$ \Gamma\Rightarrow \Delta\,||\,{\cal H}\,|\, P_i\Rightarrow Q_i\,|\, {\cal G} $}
		\DP
	\end{center}
	 is height-preserving admissible.
	\end{Lem}
\begin{proof}
If the premise is derivable, then the conclusion is derived as follows
		\begin{prooftree}
		\AXC{$\vdash_n  \Gamma\Rightarrow \Delta\,||\,{\cal H}\,|\, P_i\Rightarrow Q_i\,|\, P_i\Rightarrow Q_i\,|\, {\cal G}  $}
		\RL{$ {\sf  Merge}^{\sf c} $}
		\UIC{$\vdash_n  \Gamma\Rightarrow \Delta\,||\,{\cal H}\,|\, P_i,P_i\Rightarrow Q_i,Q_i\,|\, {\cal G} $}
		\RL{$ {\sf RC}^{\sf c} $, $ {\sf LC} ^{\sf c}$}
		\UIC{$\vdash_n  \Gamma\Rightarrow \Delta\,||\,{\cal H}\,|\, P_i\Rightarrow Q_i\,|\, {\cal G} $}
	\end{prooftree}
this rule is height-preserving admissible because the rules $ {\sf  Merge}^{\sf c} $, $ {\sf RC}^{\sf c} $, and $ {\sf LC} ^{\sf c}$ are height-preserving admissible.
\end{proof}
\section{Admissibility of cut }\label{sec cut}
In this section, we prove the admissibility of  cut rule and  completeness theorem.

In   cut rule, 
\begin{prooftree}
	\AXC{$ \Gamma\Rightarrow \Delta,D\,||\, {\cal H} $}
	\AXC{$ D,\Gamma'\Rightarrow \Delta'\,||\, {\cal H'} $}
	\RL{\sf Cut,}
	\BIC{$ \Gamma,\Gamma'\Rightarrow \Delta,\Delta'\,||\, {\cal H}\,|\, {\cal H'} $}
	\end{prooftree}
the crown of the conclusion is the union  of  crowns of the premises.

If cut formula $ D $ is atomic, then the admissibility of the cut rule
is proved  simultaneously with the following rule which is called crown cut:
\begin{prooftree}
	\AXC{$ M,P\Rightarrow Q,N\,||\, {\cal H}\,|\, P_i\Rightarrow Q_i,p\,|\,{\cal G} $}
	\RL{$ {\sf Cut}^{\sf c} $}
	\AXC{$ M',P'\Rightarrow Q',N'\,||\, {\cal H}'\,|\,p, P'_i\Rightarrow Q'_i\,|\,{\cal G}' $}
	\BIC{$ M,P_i,M',P'_i\Rightarrow Q_i,N,Q'_i,N'\,||\, {\cal H}\,|\, P\Rightarrow Q\,|\,{\cal G}\,|\, {\cal H}'\,|\, P'\Rightarrow Q'\,|\,{\cal G}' $}
\end{prooftree}
 \textbf{Note}: The following exchanges between root and crown parts should be noted in the crown cut rule $ {\sf Cut}^{\sf c} $.
 \begin{tasks}(4)
 	\task[$ \bullet $] $ P $ and $ P_i $
 	\task [$ \bullet $] $ Q $ and  $ Q_i $
 	\task [$ \bullet $] $P' $ and $ P'_i $
 	\task [$ \bullet $] $ Q' $ and  $ Q'_i $
 \end{tasks}
	\begin{Lem}\label{atomic cut}
		The following rules are admissible,	where $ p$ is an atomic formula.
		\begin{prooftree}
			\AXC{$ \Gamma\Rightarrow \Delta,p\,||\, {\cal H} $}
			\AXC{$ p,\Gamma'\Rightarrow \Delta'\,||\, {\cal H'} $}
			\RL{\sf Cut,}
			\BIC{$ \Gamma,\Gamma'\Rightarrow \Delta,\Delta'\,||\, {\cal H}\,|\, {\cal H'} $}
		\end{prooftree}
	\begin{prooftree}
		\AXC{$ M,P\Rightarrow Q,N\,||\, {\cal H}\,|\, P_i\Rightarrow Q_i,p\,|\,{\cal G} $}
		\RL{$ {\sf Cut}^{\sf c} $}
		\AXC{$ M',P'\Rightarrow Q',N'\,||\, {\cal H}'\,|\,p, P'_i\Rightarrow Q'_i\,|\,{\cal G}' $}
		\BIC{$ M,P_i,M',P'_i\Rightarrow Q_i,N,Q'_i,N'\,||\, {\cal H}\,|\, P\Rightarrow Q\,|\,{\cal G}\,|\, {\cal H}'\,|\, P'\Rightarrow Q'\,|\,{\cal G}' $}
	\end{prooftree}
		\end{Lem}
	\begin{proof}
		 Both rules are proved  simultaneously  by   induction on the sum of heights of derivations of the two premises, which is called cut-height. First, we consider  the rule $ {\sf Cut} $. If both of the premises are  initial sequents, then the conclusion is an initial sequent too, and if only one of the premises is an initial sequent, then the conclusion is obtained by  weakening. \\
	If one of the last rules in the derivations of the premises is  not 
	{\sf L$\Diamond $}, 
	{\sf R$\Box $}, or {\sf Exch},
	then the cut rule can be transformed into  cut(s) with lower
	cut-height as usual. Thus we  consider  cases which the last rules are {\sf L$\Diamond $}, 
	{\sf R$\Box $}, and {\sf Exch}.\\
	\textbf{Case 1.  The left premise is derived by {\sf R$ \Box $}.} Let $\Gamma=M,P $ and  $\Delta=Q,N,\Box A $, and let $ \Box A $ be the principal formula. We have three subcases according to the last rule in the derivation of the right premise. \\
	\textbf{Subcase 1.1} The right premise is derived by {\sf L$ \Diamond $}. Let $\Gamma'=\Diamond B,M',P'$
	and $ \Delta '=Q',N'$,  and let the last rules be as follows
	\begin{prooftree}
		\AXC{$ \D $}
		\noLine
		\UIC{$ M\Rightarrow N,A\,||\,{\cal H}\,|\,P\Rightarrow Q,p  $}
		\RL{\sf R$\Box $}
		\UIC{$ M,P\Rightarrow Q,N,\Box A,p\,||\, {\cal H}$}
		\AXC{$ \D' $}
		\noLine
		\UIC{$  B,M'\Rightarrow N' \,||\, {\cal H}'\,|\,p,P'\Rightarrow Q' $}
		\RL{\sf L$\Diamond $}
		\UIC{$ p,\Diamond B,M',P'\Rightarrow Q',N'\,||\, {\cal H}' $}
		\RL{{\sf Cut}.}
		\BIC{$  M,P, \Diamond B,M',P'\Rightarrow Q,N,\Box A,Q',N'\,||\, {\cal H}\,|\,{\cal H}' $}
	\end{prooftree}
	For this cut, let
	$ \bigwedge_{i=1}^{k}(\bigvee\neg P_i\vee \bigvee Q_i\vee \bigvee\neg M_i\vee \bigvee N_i) $ 
	be an equivalent formula in CQNF of $ A $. Then, by Corollary \ref{left to right1}, which is height-preserving admissible, we get the following derivation from $ \D $ 
	\begin{prooftree}
		\AXC{$ \D_1 $}
		\noLine
		\UIC{$ M_i,P_i,M\Rightarrow N,Q_i, N_i \,||\, {\cal H}\,|\,P\Rightarrow Q,p,  $}
	\end{prooftree}
for every clause $ (\bigvee\neg P_i\vee \bigvee Q_i\vee \bigvee\neg M_i\vee \bigvee N_i) $.
	Then by applying crown cut   rule ${\sf Cut}^{\sf c} $, we get the following derivation for every clauses in CQNF of $ A $
	\begin{prooftree}
		\AXC{$ \D_1 $}
		\noLine
		\UIC{$ M_i,P_i,M\Rightarrow N,Q_i, N_i \,||\,{\cal H}\,|\,P\Rightarrow Q,p  $}
			\AXC{$ \D' $}
			\noLine
			\UIC{$   B,M'\Rightarrow N' \,||\, {\cal H}' \,|\,p,P'\Rightarrow Q' $}
			\RL{{\sf L$\Diamond $}}
			\UIC{$ \Diamond B,M'\Rightarrow N' \,||\, {\cal H}' \,|\,p,P'\Rightarrow Q' $}
		\RL{$ {\sf Cut}^{\sf c} $}
		\BIC{$ M_i,P,M,\Diamond B,M',P'\Rightarrow Q,N,N_i,Q',N'\,||\,{\cal H}\,|\,P_i\Rightarrow Q_i\,|\, {\cal H}'\,|\,P'_i\Rightarrow Q'_i $}
		\end{prooftree}
	and then  using  Lemma \ref{NFcut2}, the conclusion is obtained.\\
		\textbf{Subcase 1.2.} The right premise is derived by {\sf R$ \Box $}. Similar to  Case 1.1.\\
		\textbf{Subcase 1.3.} The right premise is derived by {\sf Exch}.  Let $ {\cal H}'={\cal G}'\,|\,P'_i\Rightarrow Q'_i\,|\,{\cal I}'  $, and let the last rules be as follows:
		\begin{prooftree}
			\AXC{$ \D $}
			\noLine
			\UIC{$ M\Rightarrow N,A\,||\,  {\cal H}\,|\,P\Rightarrow Q,p $}
			\RL{\sf R$\Box $}
			\UIC{$ M,P\Rightarrow Q,N,\Box A,p\,||\, {\cal H}$}
			\AXC{$ \D' $}
			\noLine
			\UIC{$ M',P'_i\Rightarrow Q'_i,N'\,||\,{\cal G}'\,|\,p,P'\Rightarrow Q'\,|\,{\cal I}'  $}
			\RL{\sf Exch}
			\UIC{$p,M',P'\Rightarrow Q',N'\,||\,{\cal G}'\,|\,P'_i\Rightarrow Q'_i\,|\,{\cal I}'  $}
			\RL{{\sf Cut}.}
			\BIC{$M,P,M',P' \Rightarrow  Q,N,\Box A,Q',N'\,||\, {\cal H}\,|\,{\cal G}'\,|\,P'_i\Rightarrow Q'_i\,|\,{\cal I}' $}
		\end{prooftree}
		 This cut is transformed into
		\begin{prooftree}
			\AXC{$ \D $}
			\noLine
			\UIC{$M\Rightarrow N,A\,||\,  {\cal H}\,|\, P\Rightarrow Q,p $}
			\RL{\sf R$\Box $}
			\UIC{$ M\Rightarrow N,\Box A\,||\,  {\cal H}\,|\,P\Rightarrow Q,p$} 
			\AXC{$ \D' $}
			\noLine
			\UIC{$M',P'_i\Rightarrow Q'_i,N'\,||\,{\cal G}'\,|\,p,P'\Rightarrow Q'\,|\,{\cal I}'  $}
			\RL{$ {\sf Cut}^{\sf c} $.}  
			\BIC{$M,P,M',P' \Rightarrow  Q,N,\Box A,Q',N'\,||\, {\cal H}\,|\,{\cal G}'\,|\,P'_i\Rightarrow Q'_i\,|\,{\cal I}' $}
		\end{prooftree}
\textbf{Case 2.  The left premise is derived by {\sf L$ \Diamond $}.} Similar to the case 1.\\
\textbf{Case 3.  The left premise is derived by {\sf Exch}.} Let the hypersequent $ {\cal H}$ be as ${\cal G}\,|\, P_i\Rightarrow Q_i\,|\, {\cal I}  $.   According to the last rule applied in the right premise, we have the following subcases:\\
\textbf{Subcase 3.1.} The right premise is derived by {\sf R$ \Box $} or {\sf L$ \Diamond $}. Similar to the case 1.2.\\
\textbf{Subcase 3.2.} The right premise is derived by {\sf Exch}. Let the hypersequent $ {\cal H}'$ be as ${\cal G}'\,|\, P'_i\Rightarrow Q'_i\,|\, {\cal I}'  $ and let the last rules be as
	\begin{prooftree}
		\AXC{$ \D $}
		\noLine
		\UIC{$M,P_i\Rightarrow  Q_i,N\,||\,{\cal G}\,|\, P\Rightarrow Q,p\,|\, {\cal I} $}
		\RL{\sf Exch}
		\UIC{$M,P\Rightarrow Q,N,p\,||\,{\cal G}\,|\, P_i\Rightarrow Q_i\,|\, {\cal I}  $}
		\AXC{$ \D' $}
		\noLine
		\UIC{$M',P'_i\Rightarrow Q'_i,N'\,||\,{\cal G}'\,|\, p,P'\Rightarrow Q'\,|\, {\cal I}'   $}
		\RL{\sf Exch}
		\UIC{$p,M',P'\Rightarrow Q',N'\,||\,{\cal G}'\,|\, P'_i\Rightarrow Q'_i\,|\, {\cal I}'   $}
		\RL{{\sf Cut},}
		\BIC{$M,P,M',P'\Rightarrow Q,N,Q',N'\,||\,{\cal G}\,|\, P_i\Rightarrow Q_i\,|\, {\cal I}\,|\,{\cal G}'\,|\, P'_i\Rightarrow Q'_i\,|\, {\cal I}' $}
	\end{prooftree}
	where $ \Gamma=M,P $, $ \,\Delta=Q,N $, $\, \Gamma'=M',P' $, and $ \Delta'=Q',N'\, $. This cut is transformed into
	\begin{prooftree}
		\AXC{$ \D $}
		\noLine
		\UIC{$M,P_i\Rightarrow  Q_i,N\,||\,{\cal G}\,|\, P\Rightarrow Q,p\,|\, {\cal I} $}
		\AXC{$ \D' $}
		\noLine
		\UIC{$M',P'_i\Rightarrow Q'_i,N'\,||\,{\cal G}'\,|\, p,P'\Rightarrow Q'\,|\, {\cal I}'   $}
		\RL{$ {\sf Cut}^{\sf c} $.}
		\BIC{$M,P,M',P'\Rightarrow Q,N,Q',N'\,||\,{\cal G}\,|\, P_i\Rightarrow Q_i\,|\, {\cal I}\,|\,{\cal G}'\,|\, P'_i\Rightarrow Q'_i\,|\, {\cal I}' $}
	\end{prooftree}
\textbf{Case 4.  The right premise is derived by {\sf L$ \Diamond $}, {\sf R$ \Box $}, or {\sf Exch}.} Similar to the above cases for the left premise.\\
Now we prove the admissibility of the crown cut rule $ {\sf Cut}^{\sf c} $. If the left premise is an instance of initial sequent $ {\sf L \bot} $ or $ {\sf R \top} $, then so is the conclusion. If the left premise is an instance of initial sequent {\sf Ax}, we have two cases as follows. If $ M $ and $ N $ contain a common atomic formula, then so is the conclusion. If $ P $ and $ Q $ contain a common atomic formula, then apply the rule {\sf Exch} and then apply the rules weakening to obtain deduction of  the conclusion.  If the right premise is an initial sequent or both of the premises are initial sequent, the conclusion is obtained by the same argument.\\
 The last rule applied in the premises of the crown cute rule can only be modal rules since all formulae in this rule are modal or atomic. Similar to the proof of the cut rule, we only consider the rules {\sf L$ \Diamond $}, {\sf R$ \Box $}, and {\sf Exch}.\\
	\textbf{Case 1.  The left premise is derived by L$ \Diamond $.}
	 Let $ M=\Diamond B,M_1 $ and the derivation  be as follows
	\begin{prooftree}
		\AXC{$ \D $}
		\noLine
		\UIC{$  B,M_1\Rightarrow N\,||\, {\cal H}\,|\, P_i\Rightarrow Q_i,p\,|\, {\cal G}\,|\,P\Rightarrow Q $}
		\RL{L$ \Diamond $}
		\UIC{$ \Diamond B,M_1,P\Rightarrow Q,N\,||\, {\cal H}\,|\, P_i\Rightarrow Q_i,p\,|\, {\cal G} $}
		\AXC{$ \D' $}
		\noLine
		\UIC{$M',P'\Rightarrow Q',N'\,||\, {\cal H}'\,|\,p, P'_i\Rightarrow Q'_i\,|\,{\cal G}' $}
		\RL{$ {\sf Cut}^{\sf c} $,}
		\BIC{$\Diamond B,M_1,P_i,M',P'_i\Rightarrow Q_i,N,Q'_i,N'\,||\, {\cal H}\,|\, P\Rightarrow Q\,|\, {\cal G}\,|\, {\cal H}'\,|\, P'\Rightarrow Q'\,|\,{\cal G}' $}
	\end{prooftree}
 For this case, first we transform $ B $ into $\bigvee_{j=1}^{m}(\bigwedge P''_j\vee\bigwedge \neg Q''_j\vee\bigwedge M''_j\vee\bigwedge \neg N''_j)$, an equivalent formula in DQNF for $ B $. Then by Corollary \ref{left to right1} we get $ \D_1 $ from $ \D $ as:
	\begin{prooftree}
		\AXC{$ \D_1 $}
		\noLine
		\UIC{$  P''_j,M''_j,M_1\Rightarrow N,Q''_j,N''_j\,||\, {\cal H}\,|\, P_i\Rightarrow Q_i,p\,|\, {\cal G}\,|\,P\Rightarrow Q $}
		\end{prooftree}
for every phrases in DQNF. Thus since the rules in  Corollary \ref{left to right1}  are height-preserving admissible by induction hypothesis we have
	\begin{prooftree}
		\AXC{$ \D_1 $}
		\noLine
		\UIC{$  P''_j,M''_j,M_1\Rightarrow N,Q''_j,N''_j\,||\, {\cal H}\,|\, P_i\Rightarrow Q_i,p\,|\, {\cal G} \,|\,P\Rightarrow Q$}
		\AXC{$ \D' $}
		\noLine
		\UIC{$ M',P'\Rightarrow Q',N'\,||\, {\cal H}'\,|\,p, P'_i\Rightarrow Q'_i\,|\,{\cal G}' $}
		\RL{$ {\sf Cut}^{\sf c} $.}
		\BIC{$ M''_j, M_1,P_i,M',P'_i\Rightarrow N,Q_i,N''_j,Q'_i,N'\,||\,{\cal H}\,|\,P''_j\Rightarrow Q''_j \,|\,{\cal G} \,|\,P\Rightarrow Q \,|\, {\cal H}'\,|\, P'\Rightarrow Q'\,|\,{\cal G}'$}
	\end{prooftree}  
	Therefore,  by applying Lemma \ref{NFcut2}, the conclusion is obtained.\\
	\textbf{Case 2.  The left premise is derived by {\sf R$ \Box $}.} Similar to  Case 1.\\
	\textbf{Case 3.  The left premise is derived by {\sf Exch}.} We have two subcases according to the principal formula; the cut formula $ p $ is principal formula or not.
If the cut formula is not principal, then the derivation is transformed into a derivation with lower cut-height as follows:\\
 Let the hypersequent $ {\cal G} $ be as $ P''\Rightarrow Q''\,|\,{\cal I} $  and the last rule {\sf Exch} be as follows:
		\begin{prooftree}
		\AXC{$\D $}
		\noLine
		\UIC{$ M,P''\Rightarrow Q'',N\,||\, {\cal H}\,|\, P_i\Rightarrow Q_i,p\,|\,P\Rightarrow Q\,|\,{\cal I} $}
		\RL{\sf Exch}
		\UIC{$ M,P\Rightarrow Q,N\,||\, {\cal H}\,|\, P_i\Rightarrow Q_i,p\,|\,P''\Rightarrow Q''\,|\,{\cal I} $}
		\AXC{$\D' $}
		\noLine
		\UIC{$ M',P'\Rightarrow Q',N'\,||\, {\cal H}'\,|\,p, P'_i\Rightarrow Q'_i\,|\,{\cal G}' $}
		\RL{ $ {\sf Cut}^{\sf c} $.}
		\BIC{$ M,P_i,M',P'_i\Rightarrow Q_i,N,Q'_i,N'\,||\, {\cal H}\,|\, P\Rightarrow Q\,|\,P''\Rightarrow Q''\,|\,{\cal I}\,|\, {\cal H}'\,|\, P'\Rightarrow Q'\,|\,{\cal G}' $}
	\end{prooftree}
This cut is transformed into
\begin{prooftree}
	\AXC{$\D $}
	\noLine
	\UIC{$ M,P''\Rightarrow Q'',N\,||\, {\cal H}\,|\, P_i\Rightarrow Q_i,p\,|\,P\Rightarrow Q\,|\,{\cal I} $}
	\AXC{$\D' $}
	\noLine
	\UIC{$ M',P'\Rightarrow Q',N'\,||\, {\cal H}'\,|\,p, P'_i\Rightarrow Q'_i\,|\,{\cal G}' $}
	\RL{ $ {\sf Cut}^{\sf c} $}
	\BIC{$ M,P_i,M',P'_i\Rightarrow Q_i,N,Q'_i,N'\,||\, {\cal H}\,|\, P''\Rightarrow Q''\,|\,P\Rightarrow Q\,|\,{\cal I}\,|\, {\cal H}'\,|\, P'\Rightarrow Q'\,|\,{\cal G}' $}
\end{prooftree}
Therefore, let the left premise is derive by {\sf Exch} where the cut formula $ p $ is principal. We have three subcases according  to the last rule applied in the right premise.\\
	\textbf{Subcase 3.1.} The right premise is derived by {\sf Exch}. 
	\begin{prooftree}
		\AXC{$\D $}
		\noLine
		\UIC{$ M,P_i\Rightarrow Q_i,p,N\,||\, {\cal H}\,|\, P\Rightarrow Q\,|\,{\cal G} $}
		\RL{\sf Exch}
		\UIC{$ M,P\Rightarrow Q,N\,||\, {\cal H}\,|\, P_i\Rightarrow Q_i,p\,|\,{\cal G} $}
		\AXC{$\D'$}
		\noLine
		\UIC{$ M',p,P'_i\Rightarrow Q'_i,N'\,||\, {\cal H}'\,|\, P'\Rightarrow Q'\,|\,{\cal G}' $}
		\RL{\sf Exch}
		\UIC{$ M',P'\Rightarrow Q',N'\,||\, {\cal H}'\,|\,p, P'_i\Rightarrow Q'_i\,|\,{\cal G}' $}
		\RL{ $  {\sf Cut}^{\sf c} $.}
		\BIC{$ M,P_i,M',P'_i\Rightarrow Q_i,N,Q'_i,N'\,||\, {\cal H}\,|\, P\Rightarrow Q\,|\,{\cal G}\,|\, {\cal H}'\,|\, P'\Rightarrow Q'\,|\,{\cal G}'  $}
	\end{prooftree}
	This cut  is transformed into the first cut as follows:
	\begin{prooftree}
		\AXC{$\D $}
		\noLine
		\UIC{$ M,P_i\Rightarrow Q_i,p,N\,||\, {\cal H}\,|\, P\Rightarrow Q\,|\,{\cal G} $}
		\AXC{$\D'$}
		\noLine
		\UIC{$ M',p,P'_i\Rightarrow Q'_i,N'\,||\, {\cal H}'\,|\, P'\Rightarrow Q'\,|\,{\cal G}' $}
		\RL{\sf Cut}
		\BIC{$ M,P_i,M',P'_i\Rightarrow Q_i,N,Q'_i,N'\,||\, {\cal H}\,|\, P\Rightarrow Q\,|\,{\cal G}\,|\, {\cal H}'\,|\, P'\Rightarrow Q'\,|\,{\cal G}'  $}
	\end{prooftree}
	\textbf{Subcase 3.2.} The right premise is derived by {\sf R$ \Box $}. Let $ N'=N'',\Box A $ and $ \Box A $ be the principal formula:
	\begin{prooftree}
		\AXC{$\D $}
		\noLine
		\UIC{$ M,P_i\Rightarrow Q_i,p,N\,||\, {\cal H}\,|\, P\Rightarrow Q\,|\,{\cal G} $}
		\RL{\sf Exch}
		\UIC{$ M,P\Rightarrow Q,N\,||\, {\cal H}\,|\, P_i\Rightarrow Q_i,p\,|\,{\cal G} $}
		\AXC{$\D'$}
		\noLine
		\UIC{$ M'\Rightarrow N'', A\,||\, {\cal H}'\,|\,p, P'_i\Rightarrow Q'_i\,|\,{\cal G}'\,|\,P'\Rightarrow Q' $}
		\RL{\sf R$ \Box $}
		\UIC{$ M',P'\Rightarrow Q',N'',\Box A\,||\, {\cal H}'\,|\,p, P'_i\Rightarrow Q'_i\,|\,{\cal G}' $}
		\RL{ $  {\sf Cut}^{\sf c} $.}
		\BIC{$ M,P_i,M',P'_i\Rightarrow Q_i,N,Q'_i,N'',\Box A\,||\, {\cal H}\,|\, P\Rightarrow Q\,|\,{\cal G}\,|\, {\cal H}'\,|\, P'\Rightarrow Q'\,|\,{\cal G}' $}	
	\end{prooftree}
For this cut, let
$ \bigwedge_{j=1}^{k}(\bigvee\neg P''_j\vee \bigvee Q''_j\vee \bigvee\neg M''_j\vee \bigvee N''_j) $ 
be an equivalent formula in CQNF of $ A $. Then, by Corollary \ref{left to right1}, which is height-preserving admissible, we get the following derivation from $ \D' $ 
\begin{prooftree}
	\AXC{$ \D'_1 $}
	\noLine
	\UIC{$ M''_j,P''_j,M'\Rightarrow N'',Q''_j, N''_j\,||\,{\cal H}'\,|\,p, P'_i\Rightarrow Q'_i\,|\,{\cal G}'\,|\,P'\Rightarrow Q' ,  $}
\end{prooftree}
for every clause $ (\bigvee\neg P''_j\vee \bigvee Q''_j\vee \bigvee\neg M''_j\vee \bigvee N''_j) $.
Then by applying the  rule $  {\sf Cut}^{\sf c} $, we get the following derivation for every clauses in CQNF of $ A $
	\begin{prooftree}
	\AXC{$ M,P\Rightarrow Q,N\,||\, {\cal H}\,|\, P_i\Rightarrow Q_i,p\,|\,{\cal G}$}
	\AXC{$ \D'_1 $}
	\RL{  $  {\sf Cut}^{\sf c} $}
	\BIC{$ M,P_i,M''_j,P'_i, M'\Rightarrow Q_i,N,N'',Q'_i,N''_j\,||\, {\cal H}\,|\, P\Rightarrow Q\,|\,{\cal G}\,|\, {\cal H}'\,|\, P''_j\Rightarrow Q''_j\,|\,{\cal G}'\,|\, P'\Rightarrow Q'$}	
\end{prooftree}
	Then,  by applying Lemma \ref{NFcut2}, the conclusion is obtained.\\
\textbf{Subcase 3.3.} The right premise is derived by {\sf L$ \Diamond $}. Similar to  Subcase 3.2.
	\end{proof}
Now, we prove the admissibility of the cut rule for arbitrary formula.
\begin{The}\label{cut}
	The rule of cut
\begin{prooftree}
	\AXC{$ \Gamma\Rightarrow \Delta,D\,||\, {\cal H} $}
	\AXC{$ D,\Gamma'\Rightarrow \Delta'\,||\, {\cal H'} $}
	\RL{{\sf Cut},}
	\BIC{$ \Gamma,\Gamma'\Rightarrow \Delta,\Delta'\,||\, {\cal H}\,|\, {\cal H'} $}
\end{prooftree}
where  $ D $ is an arbitrary formula,	is admissible in $ {\cal R}_{\sf S5}$.
\end{The}
\begin{proof}
	  The proof proceeds by induction on the structure of the cut formula
	 $ D $ with subinduction on the cut-height i.e., the sum of the heights of the
	 derivations of the premises.	The admissibility of  the rule for atomic cut formula follows from Lemma \ref{atomic cut}, therefore we consider cases where $ D $ is not atomic formula.
	If the cut formula is of the form
	$ \neg A $,
	$ A\wedge B $,
	$ A\vee B $,  or
	$ A\rightarrow B $,
 then using  invertibility of the propositional rules, \Cref{invertibility of propositional rules }, the cut rule can be transformed into cut rules where cut formula is reduced, i.e, cut formula is 
	$ A $ or
	$ B $. Thus it remains to consider cases, where cut formula is of the form $ \Diamond A $ or $ \Box A $. We only consider the case where the cut formula is of the form $ \Box A $; the other case is proved similarly. We distinguish the following cases:\\
	\textbf{1. Cut formula $ \Box A $ is principal in  the left premise only.}
	we consider the last rule applied to the right premise of cut. If the last rule applied is a propositional rule, then the derivation is transformed into a derivation of lower cut-height as usual. Thus we will consider modal rules.\\
	\textbf{Subcase 1.1.}  The right premise is derived by {\sf R$ \Box $}. Let $ \Gamma'=M',P' $
	 and 
	$ \Delta'=Q',N',\Box B $, and let cut rule be as follows 
	\begin{prooftree}
		\AXC{$ \D $}
		\noLine
		\UIC{$ M\Rightarrow N, A\,||\, {\cal H}\,|\,P\Rightarrow Q $}
		\RL{\sf R$\Box $}
		\UIC{$ M,P\Rightarrow Q,N, \Box A\,||\, {\cal H} $}
		\AXC{$ \D' $}
		\noLine
		\UIC{$ \Box A,M'\Rightarrow N', B \,||\, {\cal H}'\,|\,P'\Rightarrow Q'$}
		\RL{\sf R$\Box $}
		\UIC{$  \Box A,M',P'\Rightarrow Q',N', \Box B \,||\, {\cal H}'$}
		\RL{{\sf Cut},}
		\BIC{$ M,P,M',P'\Rightarrow  Q,N,Q',N',\Box B ||\, {\cal H}\,|\, {\cal H}' $}
	\end{prooftree}
	where 	$ \Gamma=M,P $ and $ \Delta=Q,N $. This cut is transformed into
	\begin{prooftree}
			\AXC{$ \D $}
		\noLine
		\UIC{$  M\Rightarrow N, A\,||\, {\cal H}\,|\,P\Rightarrow Q $}
		\RL{\sf R$\Box $}
		\UIC{$  M\Rightarrow N, \Box A\,||\, {\cal H}\,|\,P\Rightarrow Q $}
		\AXC{$ \D' $}
		\noLine
		\UIC{$ \Box A,M'\Rightarrow N', B \,||\, {\cal H}'\,|\,P'\Rightarrow Q' $}
		\RL{\sf Cut}
		\BIC{$ M,M'\Rightarrow N, N', B\,||\,{\cal H}\,|\,P\Rightarrow Q \,|\,{\cal H}'\,|\, P'\Rightarrow Q' $}
		\RL{$ {\sf  Merge}^{\sf c} $}
		\UIC{$ M,M'\Rightarrow N, N', B\,||\, {\cal H}\,|\, {\cal H}'\,|\,P,P'\Rightarrow Q,Q' $}
		\RL{{\sf R$\Box $}.}
		\UIC{$M,M',P,P'\Rightarrow Q,Q', N, N', \Box B\,||\, {\cal H}\,|\, {\cal H}' $}
	\end{prooftree}
	\textbf{Case 1.2.}  The right premise is derived by {\sf L$ \Diamond $}. This case is treated similar to the above case.\\
	\textbf{Case 1.2.}  The right premise is derived by {\sf R$ \Diamond $}. Let
	$ \Delta'=\Delta'',\Diamond B $ and  cut rule be as follows 
	\begin{prooftree}
		\AXC{$ \D $}
		\noLine
		\UIC{$ M,P\Rightarrow Q,N, \Box A\,||\, {\cal H} $}
		\AXC{$ \D' $}
		\noLine
		\UIC{$ \Box A,\Gamma'\Rightarrow \Delta'',\Diamond B,B\,||\, {\cal H}' $}
		\RL{\sf R$\Diamond $}
		\UIC{$  \Box A,\Gamma'\Rightarrow \Delta'', \Diamond B\,||\, {\cal H}' $}
		\RL{{\sf Cut},}
		\BIC{$ M,P,\Gamma'\Rightarrow  Q,N,\Delta'',\Diamond B\,\,||\, {\cal H}\,|\, {\cal H}' $}
	\end{prooftree}
	where 	$ \Gamma=M,P $ and $ \Delta=Q,N $. This cut is transformed into
	\begin{prooftree}
		\AXC{$ \D $}
		\noLine
		\UIC{$ M,P\Rightarrow Q,N, \Box A\,||\, {\cal H} $}
		\AXC{$ \D' $}
		\noLine
		\UIC{$ \Box A,\Gamma'\Rightarrow \Delta'',\Diamond B,B\,||\, {\cal H}' $}
		\RL{\sf Cut}
		\BIC{$ M,P,\Gamma'\Rightarrow Q,N, \Delta'', \Diamond B,B\,||\, {\cal H}\,|\, {\cal H}' $}
		\RL{{\sf R$\Diamond $}.}
		\UIC{$ M,P,\Gamma'\Rightarrow  Q,N,\Delta'',\Diamond B\,\,||\, {\cal H}\,|\, {\cal H}' $}
	\end{prooftree}
	\textbf{Case 1.3.}  The right premise is derived by {\sf L$ \Box $}. This case is treated similar to the above case.\\
	\textbf{Case 1.4.}  The right premise is derived by {\sf Exch}. Let $ \Gamma'=M',P' $ and
	$ \Delta'=Q',N' $, and let the hypersequent $ {\cal H}' $ be as $ {\cal G}'\,|\, P'_i\Rightarrow Q'_i\,|\,{\cal I}' $:
	\begin{prooftree}
		\AXC{$ \D $}
		\noLine
		\UIC{$ M\Rightarrow N, A\,||\,{\cal H}\,|\, P\Rightarrow Q $}
		\RL{\sf R$ \Box $}
		\UIC{$ M,P\Rightarrow Q,N, \Box A\,||\,{\cal H} $}
		\AXC{$ \D' $}
		\noLine
		\UIC{$ \Box A,M',P'_i\Rightarrow Q'_i,N'\,||\,{\cal G}'\,|\, P'\Rightarrow Q'\,|\,{\cal I}' $}
		\RL{\sf Exch}
		\UIC{$  \Box A,M',P'\Rightarrow Q',N'\,||\,{\cal G}'\,|\, P'_i\Rightarrow Q'_i\,|\,{\cal I}' $}
		\RL{{\sf Cut},}
		\BIC{$ M,P,M',P'\Rightarrow  Q,N,Q',N'\,||\,{\cal H}\,|\,{\cal G}'\,|\, P'_i\Rightarrow Q'_i\,|\,{\cal I}'  $}
	\end{prooftree}
	where 	$ \Gamma=M,P $ and $ \Delta=Q,N $. This cut is transformed into
		\begin{prooftree}
		\AXC{$ \D $}
		\noLine
		\UIC{$ M\Rightarrow N, A\,||\,{\cal H}\,|\, P\Rightarrow Q  $}
		\RL{\sf R$ \Box $}
		\UIC{$ M\Rightarrow N, \Box A\,||\,{\cal H}\,|\, P\Rightarrow Q$}
		\AXC{$ \D' $}
		\noLine
		\UIC{$ \Box A,M',P'_i\Rightarrow Q'_i,N'\,||\,{\cal G}'\,|\, P'\Rightarrow Q'\,|\,{\cal I}'  $}
		\RL{\sf Cut}
		\BIC{$ M,M',P'_i\Rightarrow N,Q'_i, N'\,||\,{\cal H}\,|\, P\Rightarrow Q\,|\,{\cal G}'\,|\, P'\Rightarrow Q'\,|\,{\cal I}' $}
		\RL{$ {\sf  Merge}^{\sf c} $}
		\UIC{$M,M',P'_i\Rightarrow N,Q'_i, N'\,||\,{\cal H}\,|\,{\cal G}'\,|\, P,P'\Rightarrow Q,Q'\,|\,{\cal I}'  $}
		\RL{\sf Exch}
		\UIC{$ M,M',P,P'\Rightarrow  N,N',Q,Q'\,||\,{\cal H}\,|\,{\cal G}'\,|\, P'_i\Rightarrow Q'_i\,|\,{\cal I}'  $}
	\end{prooftree}
\textbf{ 2. Cut formula $ \Box A $ is principal in both premises.}
  Let $ \Gamma=M,P $
	and 
	$ \Delta=Q,N $, and let cut rule be as follows 
	\begin{prooftree}
		\AXC{$ \D $}
		\noLine
		\UIC{$ M\Rightarrow N, A\,||\, {\cal H}\,|\,P\Rightarrow Q $}
		\RL{\sf R$\Box $}
		\UIC{$ M,P\Rightarrow Q,N, \Box A\,||\, {\cal H} $}
		\AXC{$ \D' $}
		\noLine
		\UIC{$ A,\Box A,\Gamma'\Rightarrow \Delta'\,||\, {\cal H}' $}
		\RL{\sf L$\Box $}
		\UIC{$  \Box A,\Gamma'\Rightarrow \Delta' \,||\, {\cal H}'$}
		\RL{{\sf Cut},}
		\BIC{$ M,P,\Gamma'\Rightarrow  Q,N, \Delta'\,||\, {\cal H}\,|\,{\cal H}'  $}
	\end{prooftree}
	 This cut is transformed into
	 \begin{prooftree}
	 	\AXC{$ M,P\Rightarrow Q,N, \Box A\,||\, {\cal H} $}
	 	\RL{ \ref{*}}
	 	\UIC{$   M,P\Rightarrow Q,N,  A\,||\, {\cal H} $}
	 	\AXC{$ M,P\Rightarrow Q,N, \Box A\,||\, {\cal H}$}
	 	\AXC{$ \D' $}
	 	\noLine
	 	\UIC{$ A,\Box A,\Gamma'\Rightarrow \Delta'\,||\, {\cal H}' $}
	 	\RL{\sf  Cut }
	 	\BIC{$ M,P,A,\Gamma'\Rightarrow Q,N,\Delta'\,||\, {\cal H}\,|\,{\cal H}' $}
	 	\RL{\sf Cut}
	 	\BIC{$ M,P,M,P,\Gamma'\Rightarrow Q,N,Q,N,\Delta'\,||\, {\cal H}\,|\, {\cal H}\,|\,{\cal H}' $}
	 	\RL{{\sf LC}, {\sf RC}, {\sf EC}}
	 	\UIC{$ M,P,\Gamma'\Rightarrow  Q,N, \Delta'\,||\, {\cal H}\,|\,{\cal H}'  $}
	 \end{prooftree}
	 \textbf{ 3. Cut formula $ \Box A $ is not principal in the left premise.}\\
	  According to the last rule in the derivation of the left premise, we have  subcases. For propositional rules the cut rule can be  transformed into a derivation with cut(s) of lower cut-height.\\
	 \textbf{Case 3.1.}  The left premise is derived by {\sf L$ \Diamond $}. Suppose  that  $ \Gamma=\Diamond C,M,P$ and $ \Delta=N,Q $ where $  \Diamond C $ is the principal formula. Again, we have the following subcases according to the last rule in  derivation of the right premise, and we only consider the modal rules.\\
	  \textbf{Subcase 3.1.1}  The right premise is derived by {\sf L$ \Diamond $}. Let
	  $ \Gamma'=\Diamond B,M',P'\,$ and  $\Delta'=Q',N' $.
	  \begin{prooftree}
	  	\AXC{$ \D $}
	  	\noLine
	  	\UIC{$  C,M\Rightarrow N, \Diamond A\,||\, {\cal H} \,|\,P\Rightarrow Q $}
	  	\RL{\sf L$\Diamond $}
	  	\UIC{$ \Diamond C,M,P\Rightarrow Q,N, \Diamond A\,||\, {\cal H} $}
	  	\AXC{$ \D' $}
	  	\noLine
	  	\UIC{$  \Diamond A, B,M'\Rightarrow N' \,||\, {\cal H}'\,|\, P'\Rightarrow Q' $}
	  	\RL{\sf L$\Diamond $}
	  	\UIC{$ \Diamond A,\Diamond B,M',P'\Rightarrow Q',N'\,||\, {\cal H}' $}
	  	\RL{{\sf Cut},}
	  	\BIC{$ \Diamond C,M,P, \Diamond B,M',P'\Rightarrow Q,N,Q',N'\,||\, {\cal H}\,|\, {\cal H}' $}
	  \end{prooftree}
	   This cut is transformed into
	  \begin{prooftree}
	  	\AXC{$ \D $}
	  	\noLine
	  	\UIC{$  C,M\Rightarrow N, \Diamond A\,||\, {\cal H} \,|\,P\Rightarrow Q $}
	  	\AXC{$ \D' $}
	  	\noLine
	  	\UIC{$ \Diamond A, B,M'\Rightarrow N' \,||\,{\cal H}'\,|\, P'\Rightarrow Q' $}
	  	\RL{\sf L$\Diamond $}
	  	\UIC{$ \Diamond A, \Diamond B,M'\Rightarrow N' \,||\,{\cal H}'\,|\, P'\Rightarrow Q'$}
	  	\RL{ \sf Cut }
	  	\BIC{$  C,M, \Diamond B,M'\Rightarrow N,N'\,||\, {\cal H} \,|\,P\Rightarrow Q\,|\,{\cal H}'\,|\, P'\Rightarrow Q' $}
	  	\RL{$ {\sf Merge}^{\sf c} $}
	  	\UIC{$   C,M, \Diamond B,M'\Rightarrow N,N'\,||\,  {\cal H}\,|\, {\cal H}'\,|\,P,P'\Rightarrow Q,Q' $}
	  	\RL{\sf L$\Diamond $.}
	  	\UIC{$ \Diamond C,M,P, \Diamond B,M',P'\Rightarrow Q,N,Q',N' \,||\, {\cal H}\,|\, {\cal H}'$}
	  \end{prooftree}
  \textbf{Subcase 3.1.2}  The right premise is derived by {\sf R$ \Box $} or {\sf Exch}. Similar to Subcase 3.1.1.\\
  \textbf{Subcase 3.1.3}  The right premise is derived by {\sf L$ \Box $}.
   Let
  $ \Gamma'=\Box B,M',P'\,$ and  $\Delta'=Q',N' $.
  \begin{prooftree}
  	\AXC{$ \D $}
  	\noLine
  	\UIC{$  C,M\Rightarrow N, \Diamond A\,||\, {\cal H}\,|\,P\Rightarrow Q $}
  	\RL{\sf L$\Diamond $}
  	\UIC{$ \Diamond C,M,P\Rightarrow Q,N, \Diamond A\,||\, {\cal H} $}
  	\AXC{$ \D' $}
  	\noLine
  	\UIC{$  \Diamond A, B, \Box B,M',P'\Rightarrow Q',N'\,||\, {\cal H}'  $}
  	\RL{\sf L$\Box $}
  	\UIC{$ \Diamond A,\Box B,M',P'\Rightarrow Q',N'\,||\, {\cal H}' $}
  	\RL{{\sf Cut},}
  	\BIC{$ \Diamond C,M,P, \Diamond B,M',P'\Rightarrow Q,N,Q',N'\,||\, {\cal H}\,|\,{\cal H}' $}
  \end{prooftree}
	 This cut is transformed into
	 \begin{prooftree}
	 	\AXC{$ \Diamond C,M,P\Rightarrow Q,N, \Diamond A\,||\, {\cal H} $}
	 	\AXC{$ \D' $}
	 	\noLine
	 	\UIC{$ \Diamond A, B, \Box B,M',P'\Rightarrow Q',N'\,||\, {\cal H}'  $}
	 	\RL{{\sf Cut},}
	 	\BIC{$ \Diamond C,M,P, B,\Box B,M',P'\Rightarrow Q,N,Q',N'\,||\, {\cal H}\,|\,{\cal H}' $}
	 	\RL{\sf L$ \Box $}
	 	\UIC{$ \Diamond C,M,P, \Box B,M',P'\Rightarrow Q,N,Q',N' \,||\, {\cal H}\,|\,{\cal H}'$}
	 \end{prooftree}
	   \textbf{Subcase 3.1.3}  The right premise is derived by {\sf R$ \Diamond $}. Similar to Subcase 3.1.2.\\
	 \textbf{Case 3.2.}  The left premise is derived by {\sf R$ \Box $}. Similar to Case 3.1.\\
	 \textbf{Case 3.2.}  The left premise is derived by {\sf Exch}. Suppose  the hypersequent   ${\cal H}$ be as $ {\cal G}\,|\, P_i\Rightarrow Q_i\,|\, {\cal I} $.   We have the following  subcases according to the last rule in  derivation of the right premise, and we only consider modal rules here. \\
	 \textbf{Subcase 3.2.1}  The right premise is derived by {\sf Exch}. Let the hypersequent $ {\cal H}' $ be as  ${\cal G}'\,|\, P'_i\Rightarrow Q'_i\,|\,{\cal I}' $,  and let the last rules be as
	 \begin{prooftree}
	 	\AXC{$ \D $}
	 	\noLine
	 	\UIC{$M,P_i\Rightarrow Q_i,N, \Diamond A\,||\, {\cal G}\,|\, P\Rightarrow Q\,|\, {\cal I} $}
	 	\RL{\sf Exch}
	 	\UIC{$M,P\Rightarrow Q,N,\Diamond A\,||\, {\cal G}\,|\, P_i\Rightarrow Q_i\,|\, {\cal I} $}
	 	\AXC{$ \D' $}
	 	\noLine
	 	\UIC{$\Diamond A,M',P'_i\Rightarrow Q'_i,N'\,||\, {\cal G}'\,|\, P'\Rightarrow Q'\,|\,{\cal I}'  $}
	 	\RL{\sf Exch}
	 	\UIC{$\Diamond A,M',P'\Rightarrow Q',N'\,||\, {\cal G}'\,|\, P'_i\Rightarrow Q'_i\,|\,{\cal I}' $}
	 	\RL{{\sf Cut},}
	 	\BIC{$M,P,M',P' \Rightarrow Q,N,Q',N'\,||\, {\cal G}\,|\, P_i\Rightarrow Q_i\,|\, {\cal I}\,|\, {\cal G}'\,|\, P'_i\Rightarrow Q'_i\,|\,{\cal I}' $}
	 \end{prooftree}
	 where $ \Gamma=M,P $, $ \,\Delta=Q,N $, $\, \Gamma'=M',P' $, and $ \Delta'=Q',N'$. This cut is transformed into
	 \begin{prooftree}
	 	\AXC{$ \D $}
	 	\noLine
	 	\UIC{$M,P_i\Rightarrow Q_i,N, \Diamond A\,||\, {\cal G}\,|\, P\Rightarrow Q\,|\, {\cal I} $}
	 	\AXC{$ \D' $}
	 	\noLine
	 	\UIC{$ \Diamond A,M',P'_i\Rightarrow Q'_i,N'\,||\, {\cal G}'\,|\, P'\Rightarrow Q'\,|\,{\cal I}'  $}
	 	\RL{\sf Exch}
	 	\UIC{$ \Diamond A,M'\Rightarrow N'\,||\,P'_i\Rightarrow Q'_i\,|\, {\cal G}'\,|\, P'\Rightarrow Q'\,|\,{\cal I}'   $}
	 	\RL{{\sf Cut},}
	 	\BIC{$M,P_i,M'\Rightarrow Q_i,N,N',||\, {\cal G}\,|\, P\Rightarrow Q\,|\, {\cal I}\,|\,P'_i\Rightarrow Q'_i\,|\, {\cal G}'\,|\, P'\Rightarrow Q'\,|\,{\cal I}' $}
	 	\RL{${\sf Merge}^{\sf c}$}
	 	\UIC{$ M,P_i,M'\Rightarrow Q_i,N,N',||\, {\cal G}\,|\, P,P'\Rightarrow Q,Q'\,|\, {\cal I}\,|\,P'_i\Rightarrow Q'_i\,|\, {\cal G}'\,|\,{\cal I}' $}
	 	\RL{{\sf Exch}.}
	 	\UIC{$ M,P,M',P' \Rightarrow Q,N,Q',N'\,||\, {\cal G}\,|\, P_i\Rightarrow Q_i\,|\, {\cal I}\,|\, {\cal G}'\,|\, P'_i\Rightarrow Q'_i\,|\,{\cal I}' $}
	 \end{prooftree}
	 \textbf{Subcase 3.2.2}  The right premise is derived by {\sf R$ \Box $}. Similar to Subcase 3.2.1.\\
	  \textbf{Subcase 3.2.3}  The right premise is derived by {\sf R$ \Diamond $} or {\sf L$ \Box $}. Similar to Subcase 3.1.3.\\
\end{proof}
\begin{The}\label{complete}
	The following are  equivalent.
	\begin{itemize}
		\item[{\rm (1)}]
		The sequent
		$ \Gamma\Rightarrow A $ is S5-valid.
		\item[{\rm (2)}]
	$ \Gamma\vdash_{\text{S5}} A $.
		\item[{\rm (3)}]
		The sequent
	$ \Gamma\Rightarrow  A $
	is provable in $ \cal{R}_{\sf S5}  $.
\end{itemize}
\end{The}
\begin{proof}
	 (1) implies  (2)  by completeness of S5.
	 (3) implies (1)  by soundness of $ \cal{R}_{\sf S5}  $. We show that (2) implies (3).
	Suppose 
	$ A_1,\ldots,A_n $
	is  an S5-proof  of $ A $ from $ \Gamma $. This means that $ A_n $ is $ A $ and that each $ A_i $ is in $ \Gamma $, is an axiom, or is inferred by  modus ponens or necessitation. It is straightforward to
	prove, by induction on $ i $, that $\vdash \Gamma\Rightarrow A_i $ for each $ A_i $.
	
	Case 1. $ A_i\in \Gamma $: Let $ \Gamma= A_i,\Gamma' $. It can be easily proved that $ A_i,\Gamma'\Rightarrow A_i $ is derivable in $ \cal{R}_{\sf S5}  $ by induction on the complexity of $ A_i $ and using weakening rules.
	
	Case 2. $ A_i $ is an axiom of S5: All axioms of S5 are easily proved in $ \cal{R}_{\sf S5}  $. As a typical example, in the following we prove the axiom 5:
	\begin{prooftree}
		\AXC{$  A\Rightarrow A,\Diamond A $}
		\RL{$ \sf R\Diamond $}
		\UIC{$  A\Rightarrow \Diamond A $}
		\RL{$ \sf L\Diamond $}
		\UIC{$ \Diamond A\Rightarrow \Diamond A $}
		\RL{\sf R$\Box $}
		\UIC{$ \Diamond A\Rightarrow \Box\Diamond A $}
		\RL{{\sf R$\rightarrow $}.}
		\UIC{$ \Rightarrow \Diamond A\rightarrow \Box\Diamond A $}
		\end{prooftree}
	
	Case 3. $ A_i $ is inferred by  modus ponens: Suppose  $ A_i $ is inferred from $ A_j $ and $  A_j\rightarrow A_i $, $ j<i $, by use of the cut rule we prove $ \Gamma\Rightarrow A_i $:
	\begin{prooftree}
		\AXC{}
		\RL{IH}
		\UIC{$ \Gamma\Rightarrow A_j $}
		               \AXC{}
		               \RL{ IH }
		               \UIC{$ \Gamma\Rightarrow A_j\rightarrow A_i $}
		                            \AXC{$ A_j\Rightarrow A_i, A_j $}
		                             \AXC{$  A_i\Rightarrow A_i $}
		                             \RL{\sf L$\rightarrow $}
		                             \BIC{$ A_j\rightarrow A_i, A_j\Rightarrow A_i $}
		               \RL{\sf Cut}
		               \BIC{$A_j,\Gamma\Rightarrow A_i  $}
        \RL{\sf Cut}
        \BIC{$ \Gamma,\Gamma\Rightarrow A_i $}
        \RL{{\sf LC}. }
        \UIC{$ \Gamma\Rightarrow A_i $}
		\end{prooftree}
	Case 4. $ A_i $ is inferred by  necessitation:
Suppose $ A_i=\Box A_j $ is inferred from $ A_j $ by  necessitation. In this case, $ \vdash_{S5}A_j $ (since the rule necessitation  can be applied only to premises which are derivable in the axiomatic system) and so we have:
\begin{prooftree}
\AXC{}
\RL{IH}
\UIC{$ \Rightarrow A_j $}
\RL{\sf R$ \Box $}
\UIC{$ \Rightarrow\Box A_j $}
\RL{{\sf LW}.}
\UIC{$ \Gamma\Rightarrow A_i $}	
\end{prooftree}
	\end{proof}
\begin{Cor}
$ \cal{R}_{\sf S5}  $ is sound and complete with respect to the S5-{\rm Kripke} frames.
\end{Cor}
\section{Concluding Remarks}\label{conclution}
We have presented the system $ \cal{R}_{\sf S5}  $, using  rooted hypersequent, a sequent-style calculus for S5 which  enjoys the subformula property. We have proved the soundness and completeness theorems, and the  admissibility of the  weakening, contraction and cut rules in the system. In the first draft of this paper, we wrote the rules  {\sf L$ \Diamond $} and {\sf R$ \Box $} as follows:
\begin{center}
	\AXC{$ M,A\Rightarrow\Box\bigvee(\neg P,Q),N $}
	\RL{\sf L$ \Diamond $}
	\UIC{$ \Diamond A,M,P\Rightarrow Q,N $}
	\DP
	$ \qquad $
	\AXC{$ M\Rightarrow\Box\bigvee(\neg P,Q),N,A $}
	\RL{{\sf R$ \Box $}.}
	\UIC{$ M, P\Rightarrow Q,N,\Box A $}
	\DP
	\end{center}
In these rules we can use   $ \Diamond\bigwedge(P,\neg Q) $ in the antecedent  instead of  $ \Box\bigvee(\neg P,Q) $ in the succedent of the premises, since  these formulae are equivalent and  have the same role in derivations as storages;  they  equivalently can be exchanged, or be taken both of them. Taking each of them, one can prove the admissibility of the others.
By applying these rules in a backward proof search,  although  the formulae in the premises are constructed from atomic formulae in the conclusions,  the subformula property does not hold. Then, we decided to use semicolon in the middle part of the sequents.  Using extra connective semicolon ($ ; $) we introduced  a new sequent-style  calculus  for S5,  and  called it $ \text{G3{\scriptsize S5}}^; $.   Sequents in $ \text{G3{\scriptsize S5}}^; $ are  of the form
$ \Gamma;P_1;\ldots;P_n\Rightarrow Q_n;\ldots;Q_1;\Delta $, where  $ \Gamma $ and $ \Delta $ are multisets of arbitrary formulae, and $ P_i $ and $ Q_i $ are multisets of atomic formulae which serve as storages. For convenience, we used $ H $ and $ G $ to denote the sequence of multisets $P_1;\ldots;P_n$ and $ Q_n;\ldots;Q_1$, respectively. Thus, we used $ \Gamma;H\Rightarrow G;\Delta  $ to denote the sequents.
The main idea for constructing this sequent is to take an ordinary sequent $ \Gamma\Rightarrow\Delta $ as a
root  and add  two sequences of multisets of atomic formulae to it.
The system $ \text{G3{\scriptsize S5}}^; $
is obtained by extending G3c for propositional logic with the following rules:
\begin{align*}
&\AXC{$ A,\Box A,\Gamma;H\Rightarrow G; \Delta $}
\RL{L$\Box $}
\UIC{$ \Box A,\Gamma;H\Rightarrow G;\Delta $}
\DP
&&
\AXC{$ M;P;H\Rightarrow G;Q;N,A$}
\RL{\sf R$\Box $}
\UIC{$M,P;H\Rightarrow G;Q,N,\Box A $}
\DP
\\[0.15cm]
&\AXC{$  A,M;P;H\Rightarrow G;Q;N $}
\RL{L$\Diamond $}
\UIC{$\Diamond A,M,P;H\Rightarrow G;Q,N $}
\DP
&&
\AXC{$ \Gamma;H\Rightarrow G; \Delta,\Diamond A,A $}
\RL{\sf R$\Diamond $}
\UIC{$ \Gamma;H\Rightarrow G;\Delta,\Diamond A $}
\DP
\end{align*}
\begin{prooftree}
	\AXC{$M,P_i;H_1;P;H_2\Rightarrow G_2;Q;G_1;Q_i,N $}
	\RL{\sf Exch}
	\UIC{$M,P;H_1;P_i;H_2\Rightarrow G_2;Q_i;G_1;Q,N  $}
\end{prooftree}
where $M$ and $N $ are multisets of modal formulae. In backward proof  search, by applying the rules L$ \Diamond $,   {\sf R$ \Box $}, and {\sf Exch} atomic formulae in $ P $ and $ Q $ in the conclusions move to the middle (storage) parts (between two semicolons) in the premises.
These formulae, which are stored in the middle parts of sequents  until applications of the rule {\sf Exch} in a derivation, are called  related formulae, and $ (P,Q) $  is called a related pair of multisets. By applying the rule {\sf Exch},  related formulae in $ P_i $ and $ Q_i $ come out from the middle part. In other words,  multisets  $ P_i $ and $ Q_i $ exist together from storages by applications of the rule {\sf Exch}, if they have entered them together previously by applications of the rules {\sf L$ \Diamond $},  {\sf R$ \Box $}, or {\sf Exch}.

In the following, we prove  a simple sequent to show details of this system. Below, $ H=G=\emptyset $ and so for convenience we omit their semicolons, also $  P=\emptyset $ and $ Q=p $ in the rules {\sf R$ \Box $} and {\sf Exch}:
\begin{prooftree}
	\AXC{$ p,\Box p\Rightarrow p$}
	\RL{\sf L$ \Box $}
	\UIC{$ \Box p\Rightarrow p $}
	\RL{\sf Exch}
	\UIC{$ \Box p;\Rightarrow p; $}
	\RL{\sf R$ \neg $}
	\UIC{$ ;\Rightarrow p; \neg \Box p $}
	\RL{\sf R$ \Box $}
	\UIC{$ \Rightarrow p,\Box\neg \Box p $}
\end{prooftree}
This system   has the subformula property,   and we showed that all rules of this calculus are invertible and that the  rules of weakening, contraction, and  cut 
are admissible. Soundness and completeness are established as well.

The intended interpretation of the sequent $ \Gamma;P_1;\ldots;P_n\Rightarrow Q_n;\ldots;Q_1;\Delta $  is defined as follows
\[\bigwedge\Gamma\rightarrow \bigvee\Delta\vee\bigvee_{i=1}^n\Box(\bigwedge P_i\rightarrow \bigvee Q_i).\]
This interpretation is similar to  the standard formula interpretation of both nested sequents (\cite{brunnler2009deep}) and  grafted  hypersequents (\cite{Kuznets Lellmann}).  A nested sequent is a structure 
$ \Gamma\Rightarrow\Delta, [\mathcal{N}_1],\ldots,[\mathcal{N}_n] $, where $ \Gamma\Rightarrow\Delta $ is
an ordinary sequent and each $ \mathcal{N}_i $ is again a nested sequent.
A grafted  hypersequent is a structure
$ \Gamma\Rightarrow \Delta\parallel \Gamma_1\Rightarrow \Delta_1\,|\cdots|\Gamma_n\Rightarrow \Delta_n $, where $ \Gamma\Rightarrow \Delta $ is called root or  trunk, and  each $ \Gamma_i\Rightarrow \Delta_i $ is called a component.
The standard formula interpretation of nested sequent    is given recursively as
\[ (\Gamma\Rightarrow\Delta, [\mathcal{N}_1],\ldots,[\mathcal{N}_n])^{tr}:=\bigwedge \Gamma\rightarrow\bigvee\Delta \vee \Box(\mathcal{N}_1)^{tr}\vee\ldots\vee  \Box(\mathcal{N}_n)^{tr},\]
where $(\mathcal{N}_i)^{tr}$ is the standard formula interpretation of the nested sequent $ \mathcal{N}_i $. A grafted  hypersequents is essentially the same as the nested
sequent $ \Gamma\Rightarrow\Delta,[\Gamma_1\Rightarrow \Delta_1],\ldots, [\Gamma_n\Rightarrow \Delta_n] $, and the interpretation of grafted hypersequents is adapted
from the nested sequent setting as well.
Therefore, the sequent $ \Gamma;P_1;\ldots;P_n\Rightarrow Q_n;\ldots;Q_1;\Delta $,  can also be transformed into  rooted hypersequent 
$$ \Gamma\Rightarrow \Delta\parallel P_1\Rightarrow Q_1\,|\cdots|P_n\Rightarrow Q_n, $$  where all formulae in each component $ P_i\Rightarrow Q_i $ are atomic formulae.  
Therefore, we decided to rewrite the sequent  in the framework of grafted hypersequent and we call it rooted hypersequent.\\

\textbf{Acknowledgement}. The authors would like to thank Meghdad Ghari for useful
suggestions.

\newlength{\bibitemsep}\setlength{\bibitemsep}{.2\baselineskip plus .05\baselineskip minus .05\baselineskip}
\newlength{\bibparskip}\setlength{\bibparskip}{0pt}
\let\oldthebibliography\thebibliography
\renewcommand\thebibliography[1]{%
	\oldthebibliography{#1}%
	\setlength{\parskip}{\bibitemsep}%
	\setlength{\itemsep}{\bibparskip}%
}


\addcontentsline{toc}{section}{References}
\begin{thebibliography}{99}
{\small
	 \bibitem{Avron hyper}
	A. Avron. The method of hypersequents in the proof theory of propositional non-classical
	logics. In Logic: From Foundations to Applications. W. Hodges, M. Hyland, C. Steinhorn,
	and J. Truss, eds, Clarendon Press, 1996.
	
	 \bibitem{Bednarska Indrzejczak}
	 K. Bednarska and A. Indrzejczak. Hypersequent calculi for S5: The methods of cut elimination.
	 \textit{Logic and Logical Philosophy}, 24 (2015): 277–311.
	 
   \bibitem{Chellas}
   F.B. Chellas,  \textit{ Modal Logic: An Introduction}, Cambridge University Press, 1980.
	\bibitem{belnap1982display}
	\href{https://doi.org/10.1007/BF00284976}{N.D. Belnap,  Display logic, \textit{Journal of Philosophical Logic}, 11(4) (1982): 375-417.}

  \bibitem{Blackburn}
  P. Blackburn, M.De Rijke, Y.  Venema, \textit{Modal Logic}, Cambridge University Press, 2001.

	\bibitem{brauner2000cut}
	\href{https://doi.org/10.1093/jigpal/8.5.629}{T. Bra\"uner, A cut-free Gentzen formulation of the modal logic S5,\textit{ Logic Journal of IGPL}, 8(5) (2000): 629-643.}
	
	\bibitem{brunnler2009deep}
	\href{https://doi.org/10.1007/s00153-009-0137-3}{K. Br\"unnler, Deep Sequent Systems for Modal Logic, \textit{Archive for Mathematical Logic}, 48(6) (2009): 551-577.}
	
	\bibitem{Buss}
 S. R. Buss, An Introduction to Proof Theory, in: \textit{Handbook of Proof Theory}, Elsevier Science, (1998) pp. 1–78.
	
	\bibitem{Fitting}
	M. Fitting,  Proof Methods for Modal and Intuitionistic Logics, Reidel:
	Dordrecht, 1983.
	
	\bibitem{indrzejczak1998cut}
	\href{https://doi.org/10.1093/jigpal/6.3.505}{A. Indrzejczak, Cut-Free Double Sequent Calculus for S5, \textit{Logic Journal of IGPL}, 6(3) (1998): 505-516.}
	
	\bibitem{kurokawa2013hypersequent} 
	\href{https://doi.org/10.1007/978-3-319-10061-6_4}{H. Kurokawa,  \textit{Hypersequent Calculi for Modal Logics Extending S4}. In: Y. Nakano, K. Satoh, D. Bekki,  (eds.) JSAI-isAI 2013. LNCS, vol. 8417, pp. 51–68. Springer, Heidelberg, 2014.}
	
	\bibitem{Kuznets Lellmann}
	R. Kuznets and B. Lellmann. Grafting hypersequents onto nested sequents. \textit{Logic Journal of
	the IGPL}, 24 (2016): 375–423.
	
	\bibitem{Lahav}
	 O. Lahav. From frame properties to hypersequent rules in modal logics. In LICS, 2013.
	 
	 \bibitem{Lellmann}
	 B. Lellmann. Hypersequent rules with restricted contexts for propositional modal logics. \textit{Theoretical Computer Science}, 656 (2016): 76–105, 
	 
	 \bibitem{Lellmann-Pattinson}
	 B. Lellmann and D. Pattinson. Correspondence between modal Hilbert axioms and sequent
	 rules with an application to S5. In Automated Reasoning with Analytic Tableaux and Related
	 Methods, 22nd International Conference, TABLEAUX 2013, Nancy, France, September
	 16–19, 2013, Proceedings, D. Galmiche and D. Larchey-Wendling, eds, Vol. 8123 of
	 Lecture Notes in Computer Science, pp. 219–233. Springer, 2013.
	
	\bibitem{mints1974}
	G. E. Mints. Sistemy Lyuisa i sistema T (Supplement to the Russian translation). In R. Feys,
	ed., Modal Logic, pp. 422–509. Nauka, 1974.
	\bibitem{mints1997indexed} 
	\href{https://doi.org/10.1023/A:1017948105274}{G. Mints, Indexed Systems of Sequents and Cut-Elimination, \textit{Journal of Philosophical Logic}, 26(6) (1997) 671-696.}
	
	\bibitem{negri2005proof}
	\href{https://doi.org/10.1007/s10992-005-2267-3}{S. Negri,  Proof Analysis in Modal Logic. \textit{Journal of Philosophical Logic}, 34(5) (2005) 507-544.}
	
	\bibitem{negri2008structural}
S. Negri, J. Von Plato, A. Ranta,  \textit{Structural Proof Theory}. Cambridge University Press, 2008.

    \bibitem{ohnishi1959gentzen}
    \href{http://hdl.handle.net/11094/5618}{M. Ohnishi, K. Matsumoto, Gentzen Method in Modal Calculi II,\textit{ Osaka Mathematical Journal}, 11 (1959): 115-120.}
    
    \bibitem{poggiolesi2008cut}
    \href{https://doi.org/10.1017/S1755020308080040}{F. Poggiolesi, A Cut-Free Simple Sequent Calculus for Modal Logic S5, \textit{The Review of Symbolic Logic}, 1(1) (2008): 3-15.}
    
    \bibitem{poggiolesi-Gentzen Calculi}
    F. Poggiolesi, \textit{Gentzen Calculi for Modal Propositional Logic}, Vol. 32 of Trends In Logic.
    Springer, 2011.
    
    \bibitem{Pottinger}
    G. Pottinger. Uniform, cut-free formulations of T, S 4 , and S 5 (abstract). Journal of Symbolic
    Logic, 48, 900, 1983.
    \bibitem{restall2005proofnets}
    G. Restall, Proofnets for S5: Sequents and Circuits for Modal Logic, \textit{ In
    	Logic Colloquium 2005, vol. 28 of Lecture Notes in Logic}, pages 151-17, Cambridge University Press, 2007.
    
    \bibitem{sato1980cut}
    \href{https://doi.org/10.2307/2273355}{M. Sato,  A Cut-Free Gentzen-Type System for The Modal Logic S5, \textit{The Journal of Symbolic Logic}, 45(1) (1980): 67-84.}
    
   \bibitem{stouppa2007deep}
    \href{https://doi.org/10.1007/s11225-007-9028-y}{P. Stouppa,  A Deep Inference System for The Modal Logic S5,\textit{ Studia Logica}, 85(2) (2007): 199-214.}
\bibitem{bpt}
A. Troelstra, H. Schwichtenberg, Basic Proof Theory, \textit{Cambridge University Press}, Amsterdam, 1996.

    \bibitem{wansing1994sequent}
     \href{https://doi.org/10.1093/logcom/4.2.125}{H. Wansing,  Sequent Calculi for Normal Modal Propositional Logics, \textit{Journal of Logic and Computation}, 4(2) (1994): 125-142.}

     \bibitem{wansing1999predicate}
     \href{https://doi.org/10.1007/978-94-017-1280-4_12}{H. Wansing, Predicate Logics on    Display,\textit{ Studia Logica}, 62(1) (1999): 49-75.}
     
}
\end{thebibliography}
 \end{document}